\documentclass[11pt]{article}
\usepackage{graphicx}
\usepackage{amsmath,amsfonts,amsthm}
\usepackage{amssymb,latexsym}
\usepackage{fixmath}
\usepackage{mathrsfs,amsbsy}
\usepackage{dsfont}
\usepackage{enumerate}

\def\wtd{\widetilde}
\def\what{\widehat}

\DeclareMathOperator*{\argmin}{argmin}

\DeclareMathOperator{\diag}{diag}

\DeclareMathOperator{\rank}{rank}

\DeclareMathOperator{\HH}{H}
\DeclareMathOperator{\T}{T}

\def\ba{\pmb{a}}
\def\bb{\pmb{b}}
\def\bc{\pmb{c}}

\def\be{\pmb{e}}
\def\Bf{\pmb{f}}

\def\bp{\pmb{p}}
\def\bq{\pmb{q}}

\def\bw{\pmb{w}}
\def\bx{\pmb{x}}
\def\by{\pmb{y}}
\def\bz{\pmb{z}}

\newtheorem{proposition}{Proposition}[section]
\newtheorem{theorem}{Theorem}[section]

\newtheorem{corollary}{Corollary}[section]

\theoremstyle{definition}

\newtheorem{remark}{Remark}[section]

\numberwithin{equation}{section}
\numberwithin{figure}{section}
\numberwithin{table}{section}

\allowdisplaybreaks

\usepackage{kbordermatrix}

\usepackage{caption}
\usepackage{amsfonts,amsthm}
\usepackage{amsmath,amssymb,enumerate,color}
\usepackage{algorithm,algorithmic}
\usepackage{epsfig,graphicx,psfrag,mathrsfs,subfigure}
\usepackage{eucal}
\usepackage{yhmath}
\usepackage{hyperref}
\hypersetup{colorlinks}
\usepackage{diagbox}
\RequirePackage{multirow}
 
\usepackage{textcomp}
\usepackage[]{fontenc}
\usepackage{extarrows}
\def\wtd{\widetilde}
\def\what{\widehat}
\usepackage{rotating}
\usepackage{lscape}
\usepackage{bm}
\usepackage{xspace}
\usepackage{tikz}
\usetikzlibrary{petri} 
\usetikzlibrary{shadows,arrows,positioning,shapes.geometric} 

\usepackage{geometry}
\def\ba{\pmb{a}}
\def\bb{\pmb{b}}
\def\bc{\pmb{c}}

\def\be{\pmb{e}}
\def\Bf{\pmb{f}}

\def\bp{\pmb{p}}
\def\bq{\pmb{q}}

\def\bw{\pmb{w}}
\def\bx{\pmb{x}}
\def\by{\pmb{y}}
\def\bz{\pmb{z}}

\def\tu{\mathfrak{u}}

\def\diag{{\rm diag}}

\def\scrR{\mathscr{R}}

\def\wtd{\widetilde}
\def\what{\widehat}

\def\bbC{\mathbb{C}}

\def\bbP{\mathbb{P}}
\def\bbR{\mathbb{R}}

\renewcommand{\algorithmicrequire}{\textbf{Input:}}
\renewcommand{\algorithmicensure}{\textbf{Output:}}

\numberwithin{equation}{section}
\numberwithin{figure}{section}
\numberwithin{table}{section}

\graphicspath{{figures/}{figure/}{pictures/}%
	{picture/}{pic/}{pics/}{image/}{images/}}

\title{A convex dual {problem} for the rational minimax approximation and Lawson's iteration}
\author{Lei-Hong Zhang\thanks{Corresponding author. School of Mathematical Sciences, Soochow University, Suzhou 215006, Jiangsu, China. This work was
 supported in part by the National Natural Science Foundation of China NSFC-12471356,  NSFC-12071332, NSFC-12371380, Jiangsu Shuangchuang Project (JSSCTD202209),  Academic Degree and Postgraduate Education Reform Project of Jiangsu Province, and China Association of Higher Education under grant 23SX0403.
        Email: {\tt longzlh@suda.edu.cn}.} \and Linyi Yang\thanks{School of Mathematical Sciences, Soochow University, Suzhou 215006, Jiangsu, China. Supported in part by Postgraduate Research \& Practice Innovation Program of Jiangsu Province KYCX23\_3226. Email: {\tt lyyang161@stu.suda.edu.cn}.} \and Wei Hong Yang\thanks{School of Mathematical Sciences, Fudan University, Shanghai 200433, People's Republic of China. This work was supported  in part by the National Natural
Science Foundation of China NSFC-72394365. Email: {\tt whyang@fudan.edu.cn}.} \and Ya-Nan Zhang\thanks{School of Mathematical Sciences, Soochow University, Suzhou 215006, Jiangsu, China.  Email: {\tt ynzhang@suda.edu.cn}.}}
 
 \date{ }

\begin{document}

\maketitle

\begin{abstract}
Computing the discrete rational minimax approximation in the complex plane is challenging. Apart from Ruttan's sufficient  condition, there are few other sufficient conditions for global optimality. The state-of-the-art rational approximation algorithms, such as the adaptive Antoulas-Anderson (AAA), AAA-Lawson, and the rational Krylov fitting (RKFIT) method, perform highly efficiently, but the computed rational approximations may not be minimax solutions. In this paper, we propose a convex programming approach, the solution of which is guaranteed to be the   rational minimax approximation under Ruttan's sufficient condition. Furthermore, we present a new version of Lawson's iteration for solving this convex programming problem. The computed solution can be easily verified as the   rational minimax approximation. Our numerical experiments demonstrate that this updated version of Lawson's iteration generally converges monotonically with respect to the objective function of the convex optimization. It is  an effective competitive approach for  computing the rational minimax approximation, compared to the highly efficient AAA, AAA-Lawson,  and the stabilized Sanathanan-Koerner iteration.
\end{abstract}


\medskip
{\small
\medskip

}
 
\section{Introduction}\label{sec_intro}
Let $\Omega$ be a compact subset in the complex plane $\bbC$ and $f(x)$ be  a continuous complex-valued function on $\Omega$. Denote by $\bbP_n={\rm span}(1,x,\cdots,x^n)$ the set of complex polynomials with degree less than or equal to $n$. 
For  given   data points $\{(x_j,f_j)\}_{j=1}^m$ sampled from  $f_j=f(x_j)\in \bbC$ ($x_j\in \Omega$) over distinct nodes  ${\cal X}=\{x_j\}_{j=1}^m$, we consider the discrete  rational    approximation in the form 
\begin{equation}\label{eq:bestf0}
\inf_{\xi=p/q\in\scrR_{(n_1,n_2)}}\|\Bf-\xi(\bx)\|_{\infty},
\end{equation}
where $\scrR_{(n_1,n_2)}:=\{{p}/{q}|p\in \bbP_{n_1},~0\not\equiv q\in\bbP_{n_2}\}$, 
$\Bf=[f_1,\dots,f_m]^{\T}\in\bbC^m ~(n_1+n_2+2\le m)$,  $\bx=[x_1,\dots,x_m]^{\T}\in\bbC^m$, and
\begin{equation}\nonumber 
\|\Bf-\xi(\bx)\|_{\infty}=\max_{1\le j\le m}\left|f_j-\frac{p(x_j)}{q(x_j)}\right|.
\end{equation}
When the infimum of \eqref{eq:bestf0} is attainable,  the   function $\xi^*={p^*}/{q^*}\in \scrR_{(n_1,n_2)} $ defined by
\begin{equation}\label{eq:bestf}
{p^*}/{q^*} \in \arg\min_{\xi=p/q\in\scrR_{(n_1,n_2)}}\|\Bf-\xi(\bx)\|_{\infty},
\end{equation}
is called the rational minimax (best or   Chebyshev) approximant  \cite{tref:2019a} of  $f(x)$ over ${\cal X}$.

The rational minimax approximation is a classical topic in approximation theory \cite[Chapter 24]{tref:2019a} and  can be useful in some applications (see e.g., \cite{decz:1974,gust:1983}). 
Different from the traditional polynomial minimax approximation, computing the discrete rational minimax approximation in the complex plane is challenging \cite{natr:2020}. First, it is known that the infimum of \eqref{eq:bestf0} may not be attainable, and even if it is, the   minimax approximation may not be unique \cite{natr:2020,tref:2019a}; moreover, local best   solutions may also exist \cite{natr:2020,tref:2019a}. Though the Kolmogorov condition in the primal and dual forms \cite{rutt:1985} and other types of necessary optimality conditions (see e.g., \cite{elli:1978,gutk:1983,rutt:1985,sava:1977,this:1993,will:1972,will:1979,wulb:1980}) are available for a local best approximation of \eqref{eq:bestf}, there are few effective sufficient conditions for the global optimality. To the  best of authors' knowledge,  Ruttan's optimality condition \cite[Theorems 2.1]{rutt:1985} may be the only useful global optimality of  \eqref{eq:bestf}.  
 
In recent years, a few highly efficient approaches for computing the rational approximations have been proposed. Among them, the adaptive Antoulas-Anderson (AAA) \cite{nase:2018} and the AAA-Lawson algorithm \cite{fint:2018,natr:2020} are remarkable; other efficient methods include the differential correction algorithm \cite{bapr:1972,chlo:1963} and a barycentric version of the differential correction algorithm \cite{fint:2018}, vector fitting \cite{guse:1999},  the rational Krylov fitting (RKFIT) method \cite{begu:2017,gogu:2021}, the Loeb algorithm  \cite{loeb:1957} (independently proposed by Sanathanan and  Koerner \cite{sako:1963}, known as the SK iteration) and the stabilized SK iteration \cite{hoka:2020}. Some of these methods perform highly efficiently, but the computed rational approximants may not be the minimax solutions.

Lawson's iteration \cite{laws:1961,elwi:1976} is an effective traditional method for computing discrete linear {(for example, the polynomial or trigonometric polynomial)} minimax approximations, {whose variants} (e.g., \cite{bama:1970,coop:2007,hoka:2020,loeb:1957,sako:1963}) have been proposed to solve \eqref{eq:bestf}. It is  an iteratively reweighted least-squares (IRLS) iteration. However, the underlying problem that Lawson's iteration intrinsically targets is usually hidden.
One of {the} keys, in our opinion, to design and analyze Lawson's algorithm iteration is to understand the meaning of the associated weights. We shall reveal that the hidden problem related to Lawson's iteration  is the dual (a convex max-min programming  {formulation}) of a certain reformulation of the original minimax problem, and the weights are indeed the dual variables. Such a reformulation of the original minimax problem  is called the primal problem in modern optimization theory. This relation has been described in \cite{yazz:2023} (see also \cite{clin:1972,rice:1969}) {for  linear minimax approximation.} It is possible that there are various types of equivalent reformulations of the original minimax approximation, leading to different dual {formulations}. Note that the Lagrange duality theory applying to a primal ($\min-\max$) problem   leads to a ($\max-\min$)   dual {problem} (see e.g., \cite[Chapter 5.2]{boyd:2004}), and moreover, the optimal objective function value of the dual  provides a lower bound for the optimum of the primal; this is the {\it weak duality} in Lagrange duality theory \cite[Chapter 5.2]{boyd:2004}. A {desirable} property is the so-called {\it strong duality} which means that the gap between the original primal and its dual vanishes. Relying upon the Lagrange duality, a proper reformulation of  the original minimax problem should induce a dual {problem} satisfying  strong duality. This is the case for the linear minimax problem because the original  primal ($\min-\max$)  problem is  convex, and the duality gap vanishes naturally under the Slater condition \cite[Section 5.2.3]{boyd:2004};  the classical Lawson's iteration is a {monotonic} ascent method (with respect to its dual objective function) for solving the dual  \cite{yazz:2023}.

The rational structure of the best solution $\xi^*=p^*/q^*$ complicates the rational minimax problem \eqref{eq:bestf}. First, it has been pointed out that the infimum of \eqref{eq:bestf} may not be attainable {in the given  $\scrR_{(n_1,n_2)}$ (see e.g., \cite{natr:2020})}, and even if  it {is}, {the  best approximation may not be unique \cite{sava:1977}, and local best rational approximations exist \cite{isth:1993}}. Due to the non-convexity of \eqref{eq:bestf}, a reformulation of the original minimax problem as well as its associated dual problem is not that straightforward as that in the linear case. Lawson's iteration   works and is capable of computing the   rational minimax  approximation only if there is no duality gap (i.e., strong duality holds) between its underlying primal and dual {problems}.   Previously, some IRLS iterations, including the Loeb algorithm  \cite{loeb:1957} (i.e.,  the SK iteration \cite{sako:1963}), the stabilized SK iteration \cite{hoka:2020} and the AAA-Lawson iteration \cite{fint:2018,nase:2018},  have been proposed and extended the traditional Lawson's iteration  to the rational case \eqref{eq:bestf}, but there is no discussion on their   underlying dual {problems}; furthermore, there is no convergence proof for these versions of Lawson's iteration, and {even if they do} converge, the computed solution {may not be the minimax approximants}.   We will discuss these IRLS iterations  in  more detail in section \ref{subsec_others}.

Based on the above discussion, the aim of this paper is first to establish a proper dual {problem} of the original minimax problem \eqref{eq:bestf}; particularly, even though the original   \eqref{eq:bestf} is not convex, we will show that    strong duality holds if and only if Ruttan's sufficient condition \eqref{eq:Ruttan} {is fulfilled}. Our established  dual problem is convex  over the probability simplex constraint:
\begin{equation}\label{eq:simplex}
{\cal S}:=\{\bw=[w_1,\dots,w_m]^{\T}\in \bbR^m: \bw\ge 0 ~{\rm and } ~\bw^{\T}\be=1\},~~\be=[1,1,\dots,1]^{\T},
\end{equation}
 and therefore, we can expect to find the best rational   approximation of \eqref{eq:bestf} through solving this convex dual problem under Ruttan's sufficient condition \eqref{eq:Ruttan}. With this dual {problem}, it then becomes natural to design a new version of Lawson' iteration,  a method for solving the dual with its dual variables as the corresponding weights. A more clear description of the framework of this paper is illustrated in   Figure \ref{fig:framework}.

\begin{figure}[H]  
 \tikzstyle{P} = [rectangle, rounded corners, minimum width=1cm, minimum height=1cm, text centered, text width=4.6cm, draw=black,  fill=red!30, drop shadow]
\tikzstyle{D} = [rectangle, rounded corners, minimum width=2cm, minimum height=0cm, text centered, text width=5.42cm, draw=black, fill=blue!30, drop shadow] 
\tikzstyle{M} = [rectangle, rounded corners, minimum width=2cm, minimum height=0cm, text centered, text width=3.42cm, draw=black, fill=white!, drop shadow] 

\tikzstyle{arrow} = [ -->, >=stealth]
\tikzstyle{linepart} = [draw, thick, color=black!50, -latex', dashed] 
\tikzstyle{line} = [draw,  thick,    color=red!,  -latex']
 
\newcommand{\Primal}[2]{node (p#1) [P] {#2}}
\newcommand{\Dual}[2]{node (p#1) [D] {#2}}
    \newcommand{\Method}[2]{node (p#1) [M] {#2}}
\begin{tikzpicture}[node distance=1cm,x=0.675cm,y=0.6cm]\nonumber
\path \Primal{2}{{\small Primal problem: \vskip -7mm$${  (Chebyshev ~appr.)~} {\rm min-max} $$}}; 
\path (p2.east)+(7.5,-6.0) \Dual{3}{{\small Dual problem ${\rm (max-min)}$: \vskip -7mm $$\max_{\bw\in {\cal S}}d_2(\bw)$$}};
  \path (p3.east)+(0,6) \Method{5}{{\small \vskip -8mm$${\rm Lawson's~ iteration}$$ }};  
  \draw [line,red] (p5.south)  -- node [above=1pt, sloped] {solve the dual} (p3);
\draw [line,red] (p2.east)  -- node [above=5pt, sloped] {weak duality $\checkmark$} (p3);
\draw [line,red] (p2.east)   -- node [below=-1pt, sloped] {strong duality under} (p3);
\draw [line,red] (p2.east)  -- node [below=9pt, sloped] {Ruttan's condition} (p3);
\draw (2.5,-1 ) to[bend right,   draw=graphicbackground,sloped]  node [midway] {Lagrange duality} (6.7,-6 );
\draw (2.5,-1 ) to[bend right, below=5pt, sloped]  node [midway] {on a linearization} (6.7,-6 );
\end{tikzpicture}
\caption{Framework of handing the rational minimax approximation of \eqref{eq:bestf}.}
\label{fig:framework}
 \end{figure}
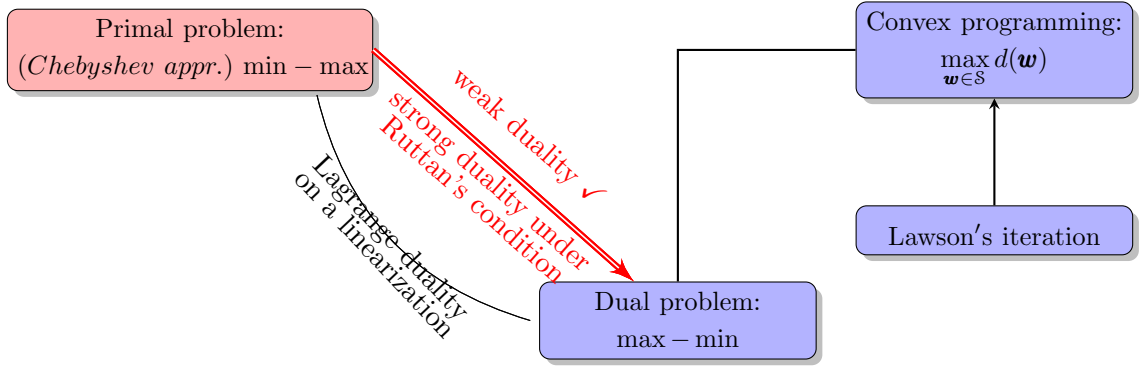

{\bf Contributions}.
We summarize our contributions of this paper based on Figure \ref{fig:framework}.
\begin{itemize}
\item[(1)] We reformulate the original minimax approximation of \eqref{eq:bestf} by a linearization as a primal programming {problem} and construct its convex dual {problem};

\item[(2)] We prove that   strong duality   between the primal and its dual holds  if and only if Ruttan's sufficient condition \eqref{eq:Ruttan} is true;
\item[(3)]  We design a new version of Lawson's iteration based on the dual {problem}, and carry out   numerical experiments on both real and complex problems; our numerical results demonstrate that the new Lawson's iteration is a very effective method to solve {the} original rational minimax problem  \eqref{eq:bestf}.
\end{itemize}

{\bf Paper organization}.  
We organize the paper as follows. In section \ref{sec_dual}, we first discuss a linearization of original rational minimax problem  \eqref{eq:bestf} to form a primal programming {problem}; we will show that under certain conditions, the   rational minimax  approximation $\xi^*$ of \eqref{eq:bestf} can be recovered from this reformulation. Our dual {problem} is based on this primal {problem} and is obtained by the classical Lagrange duality theory.  Weak duality and strong duality will be discussed.  In section \ref{sec_d2}, we shall show that computing the dual objective function value amounts to solving a Hermitian positive semi-definite matrix pencil, and efficient computational techniques will be introduced to compute the objective function value, as well as the evaluation of the computed rational approximation at new nodes. Section \ref{sec_strongduality} is devoted to  strong duality; in particular, we will prove that    strong duality is equivalent to Ruttan's sufficient condition, and thereby, the minimax rational   approximation of \eqref{eq:bestf} can be {found} via solving the  convex dual {problem} under Ruttan's sufficient condition \eqref{eq:Ruttan}. In section \ref{sec_gradHess}, further information on the gradient and the Hessian of the dual objective function are provided. Our new Lawson's iteration for solving the dual {problem} is proposed in section \ref{sec_lawson}, where we will also discuss other types of Lawson's iteration in the literature. Our numerical experiments on the new Lawson's iteration as well as its comparison to others are presented in section \ref{sec_Numerical},  and finally concluding remarks  are drawn in section \ref{sec_conclude}.

{\bf Notation}.
Throughout this paper, ${\tt i}=\sqrt{-1}$ represents the imaginary unit, and for $\mu\in \bbC$, we write $\mu= \mu^{\tt r}+ {\tt i} \mu^{\tt i}$ and $\bar\mu= \mu^{\tt r}- {\tt i} \mu^{\tt i}$ where ${\rm Re }(\mu)=\mu^{\tt r}\in \bbR,~{\rm Im }(\mu)=\mu^{\tt i}\in \bbR$ and $|\mu|=\sqrt{(\mu^{\tt r})^2+(\mu^{\tt i})^2}$. Vectors are denoted in bold lower case letters, and ${\mathbb C}^{n\times m}$ (resp. ${\mathbb R}^{n\times m}$) is the set
of all $n\times m$ complex (resp. real) matrices, with $I_n\equiv [\be_1,\be_2,\dots,\be_n]\in\bbR^{n\times n}$ representing the $n$-by-$ n$ identity matrix, where $\be_i$ is its $i$th column with $i\in [n]:=\{1,2,\dots,n\}$. We use $\diag(\bx)=\diag(x_1,\dots,x_n)$ to denote the diagonal matrix associated with the vector $\bx\in\bbC^n$, and define $\bx./\by=[x_1/y_1,\dots,x_n/y_n]$ for two vectors $\bx,\by\in \bbC^n$ with $y_j\ne 0,~1\le j\le n$.   $\|\bx\|_\alpha=(\sum_{j=1}^n|x_j|^\alpha)^\frac1\alpha$ is the vector $\alpha$-norm ($\alpha\ge 1$) of $\bx\in \bbC^n$. 
For  $A\in\bbC^{m\times n}$, $A^{\HH}$ (resp. $A^{\T}$)  and $A^{\dag}$ are the conjugate transpose (resp. transpose) and the Moore-Penrose inverse of $A$, respectively; ${\rm span}(A)$ represents the column space of $A$.  Also, we denote the  $k$th Krylov subspace generated by a matrix $A$ on $\bb$ by
\[
{\cal K}_k(A,\bb)={\rm span}(\bb,A\bb,\dots,A^{k-1}\bb).
\]

\section{A dual {formulation} of the rational minimax {approximation}} \label{sec_dual}
First, to parameterize  $p/q\in \scrR_{(n_1,n_2)}$, we let $\bbP_{n_1}={\rm span}(\psi_0(x),\dots,\psi_{n_1}(x))$ and $\bbP_{n_2}={\rm span}(\phi_0(x),\dots,\phi_{n_2}(x))$ be the chosen bases for the numerator and denominator  polynomial spaces, respectively. Thus, for any $p/q\in \scrR_{(n_1,n_2)}$, we can write 
\begin{equation}\nonumber
\frac{p(x)}{q(x)}=\frac{[\psi_0(x),\dots,\psi_{n_1}(x)] \ba}{[\phi_0(x),\dots,\phi_{n_2}(x)] \bb},~~\mbox{for ~some~} \ba\in \bbC^{n_1+1}, \bb\in \bbC^{n_2+1}.
\end{equation}  
Corresponding to the set of points ${\cal X}=\{x_j\}_{j=1}^m$ with $|{\cal X}|=m$ nodes, we have the following basis matrix for $p\in \bbP_{n_1}$:
\begin{equation}\nonumber
\Psi=\Psi(x_1,\dots,x_m;n_1):=\left[\begin{array}{cccc}\psi_0(x_1) & \psi_1(x_1) & \cdots & \psi_{n_1}(x_1) \\\psi_0(x_2) & \psi_1(x_2) & \cdots & \psi_{n_1}(x_2)  \\ \vdots & \cdots & \cdots& \vdots  \\\psi_0(x_m) & \psi_1(x_m) & \cdots & \psi_{n_1}(x_m) \end{array}\right],~\Psi_{i,j}=\psi_{j-1}(x_i),
\end{equation}
and similarly, $\Phi=\Phi(x_1,\dots,x_m;n_2)=[\phi_{j-1}(x_i)] \in \bbC^{m\times (n_2+1)}$.

Let $\xi(x)=p(x)/q(x)\in\scrR_{(n_1,n_2)}$ be irreducible.
 If $|\xi(x)|$ is bounded for any $x\in {\cal X}$, then it is easy to see that $q(x)\ne 0$ for any $x\in {\cal X}$, and we define the maximum error
\begin{equation}\label{eq:exi}
e(\xi):=\max_{x\in {\cal X}}|f(x)-\xi(x)|=\|\Bf-\xi(\bx)\|_\infty.
\end{equation}
Associated with $p/q$ is the number 
\begin{equation}\label{eq:defect}
\upsilon(p,q):=\min(n_1-\deg(p),n_2-\deg(q)),
\end{equation} which is called the {\it defect} of $p/q$, where $\deg(p)$ and $\deg(q)$ denote  the degrees of $p$ and $q$, respectively. We say $\xi(x)=p(x)/q(x)$ is {\it non-degenerate} if $\upsilon(p,q)=0$.

\subsection{A linearized reformulation}\label{subsec:linearity}

For the linear minimax approximation in the real case, a well-known equivalent reformulation is to introduce the upper bound of the derivation for each node, which is then minimized  subject to the upper bound conditions. In the real case, the resulting optimization can be expressed as a linear programming {problem} (see e.g., \cite[Section 1.2.2]{boyd:2004} and \cite{yazz:2023}). Following the same idea, for the complex case, a dual {problem} of the reformulation has been discussed in \cite{yazz:2023} and the classical Lawson's iteration \cite{laws:1961} is analyzed and shown to be a  monotonically increasing iteration for solving the dual {problem}; furthermore, a new method, the interior-point method, is proposed to solve   the linear minimax approximation for the complex case. For {computing} the rational approximation, this idea has been used for the real case \cite{dips:2022b,dips:2022a}. The following theorem {describes} a special case to motivate our proposed dual {problem}. 

\begin{theorem}\label{thm:linearity}
 Given $m\ge {n_1}+n_2+2$ distinct nodes ${\cal X}=\{x_j\}_{j=1}^m$ on $\Omega\subset\bbC$, suppose $\xi^*(x)=p^*(x)/q^*(x)\in \scrR_{(n_1,n_2)}$ is the unique solution \eqref{eq:bestf} which is irreducible and non-degenerate (i.e., $\upsilon(p^*,q^*)=0$ in \eqref{eq:defect}). Let 
$$
\eta_\infty =\|\Bf-\xi^*(\bx)\|_\infty.
$$
Then for any $\alpha\ge 1$, the triple $(\eta_\infty^\alpha,p^*,q^*)$ is also the solution to the following problem 
 \begin{align}\nonumber
&\inf_{\eta\in \bbR,~p\in \bbP_{n_1},~q\in \bbP_{n_2}\setminus\{0\}}\eta \\\label{eq:linearity}
 s.t., ~& |f_jq(x_j)-p(x_j)|^\alpha\le \eta |q(x_j)|^\alpha,~~\forall j\in [m].
\end{align}
\end{theorem}
\begin{proof}
See Appendix \ref{AppendixA}.
\end{proof}

It is worth mentioning that  Theorem \ref{thm:linearity}  only provides a sufficient condition to recover $\xi^*$ from \eqref{eq:linearity}; we defer the equivalent condition (Theorem \ref{thm:strongdualityeqvRuttan})  for obtaining $\xi^*$ (not necessarily non-degenerate) from \eqref{eq:linearity} to section \ref{sec_strongduality}.
\subsection{A dual {programming formulation}}\label{subsec:dual}
As explored in \cite{yazz:2023}, for the linear minimax approximation, the classical Lawson's iteration is a  monotonically ascent iteration for solving the dual {problem} associated with the equivalent formulation; similarly,  the linearized reformulation of  \eqref{eq:linearity} attempts to minimize the upper bound of the deviation at every given node. In the following discussion, we shall see that  the dual {problem of \eqref{eq:linearity} amount to  maximizing} the corresponding Lagrange dual function with respect to the dual variables (i.e., Lagrange multipliers) $0\le \bw=[w_1,\dots,w_m]^{\T}\in \bbR^m$, i.e.,  {\it weights} in Lawson's iteration. The purpose of this  subsection is to  derive the Lagrange dual function as well its  dual {problem} for \eqref{eq:linearity}.

To apply the traditional Lagrange duality theory (see e.g., \cite{boyd:2004,nowr:2006}), we can write $p(x_j)$ and $q(x_j)$ as $p(x_j;\ba)$ and $q(x_j;\bb)$, respectively, to  indicate the coefficient vectors $\ba=\ba^{\tt r}+ {\tt i}\ba^{\tt i}$ and $\bb=\bb^{\tt r}+ {\tt i}\bb^{\tt i}$ in $p\in \bbP_{n_1}$ and $q\in \bbP_{n_2}$, respectively. Now, introduce the Lagrange multipliers $0\le \bw=[w_1,\dots,w_m]^{\T}\in \bbR^m$ and define the Lagrange function of \eqref{eq:linearity} as
\begin{align}\nonumber
L(\eta,\ba,\bb;\bw)&=\eta- \sum_{j=1}^mw_j\left(\eta |q(x_j;\bb)|^\alpha-|f_jq(x_j;\bb)-p(x_j;\ba)|^\alpha \right)\\\nonumber
&=\eta\left(1-\sum_{j=1}^mw_j|q(x_j;\bb)|^\alpha\right)+\sum_{j=1}^mw_j |f_jq(x_j;\bb)-p(x_j\ba)|^\alpha.
\end{align}
It is easy to see that for a given $\bw\ge 0$, 
\begin{equation}\nonumber
\inf_{\eta,\ba,\bb}L(\eta,\ba,\bb;\bw)=\left\{\begin{array}{cc}-\infty,& \mbox{if~} 1\ne \sum_j^mw_j|q(x_j;\bb)|^\alpha, \\\inf_{\ba,\bb}\sum_{j=1}^mw_j |f_jq(x_j;\bb)-p(x_j\ba)|^\alpha, & \mbox{if~} 1=\sum_j^mw_j|q(x_j;\bb)|^\alpha.\end{array}\right.
\end{equation}
Therefore, the associated Lagrange dual function is given by 
\begin{equation}\label{eq:Lagrangedualfun2}
\inf_{\begin{subarray}{c} \ba\in \bbC^{n_1+1},~\bb\in \bbC^{n_2+1}\\
            \sum_{j=1}^m w_j |q(x_j;\bb)|^\alpha=1\end{subarray}}\sum_{j=1}^m w_j |f_j q(x_j;\bb)-p(x_j;\ba)|^\alpha,~~0\le \bw\in\bbR^m.
\end{equation}
{As a practical choice of $\alpha\ge 1$, we set $\alpha=2$ in our following discussion.} 
The next theorem asserts that the weak duality of the dual {problem holds, i.e., the objective function value of the dual  at any $\bw\ge 0$ will always provide a lower bound for $\eta_\infty$.}
\begin{theorem}\label{thm:q-dual}
 Given $m\ge {n_1}+n_2+2$ distinct nodes ${\cal X}=\{x_j\}_{j=1}^m$ on $\Omega$, suppose $\xi^*=p^*/q^*\in \scrR_{(n_1,n_2)}$ with $q^*(x_j)\ne 0~\forall j\in[m]$ is a solution to \eqref{eq:bestf}. Let 
$$
\eta_\infty=\|\Bf-\xi^*(\bx)\|_\infty.
$$
Then we have the weak duality:
\begin{equation}\label{eq:rat-dual}
    \sup_{\bw\ge \mathbf{0}}d_2(\bw) =\max_{\bw\in {\cal S}}d_2(\bw)\le (\eta_\infty)^2, 
\end{equation}
where ${\cal S}$ is the probability simplex given in \eqref{eq:simplex} and 
\begin{equation}\label{eq:rat-d}
d_2(\bw)=\min_{\begin{subarray}{c}p\in \bbP_{n_1},~q\in \bbP_{n_2}\\
            \sum_{j=1}^m w_j |q(x_j)|^2=1\end{subarray}}\sum_{j=1}^m w_j |f_j q(x_j)-p(x_j)|^2.
\end{equation}  
\end{theorem}
\begin{proof}

We  prove that 
\begin{equation}\nonumber
d_2 (\bw)\le (\eta_\infty)^2,\quad \forall \bw\ge \mathbf{0},
\end{equation} 
and therefore, $\sup_{\bw\ge \mathbf{0}}d_2(\bw) \le (\eta_\infty)^2$. 
The result is trivial if $\bw=\mathbf{0}$.
 For a nonzero $\bw\ge \mathbf{0}$,  use the bases matrices $\Psi$ and $\Phi$, and  
$$
\sum_{j=1}^m w_j |q(x_j)|^2=\|\sqrt{W}\Phi \bb\|_2^2,~~\sum_{j=1}^m w_j |f_j q(x_j)-p(x_j)|^2=\left\|\sqrt{W}[-\Psi,F\Phi]  \left[\begin{array}{c}\ba \\\bb\end{array}\right]\right\|_2^2
$$ 
to rewrite \eqref{eq:rat-d} as
\begin{equation}\label{eq:rat-d-compt}
d_2(\bw)=\min_{\begin{subarray}{c}  \ba\in \bbC^{n_1+1},~ \bb\in \bbC^{n_2+1}\\
            \|\sqrt{W}\Phi \bb\|_2 =1\end{subarray}}\left\|\sqrt{W}[-\Psi,F\Phi]  \left[\begin{array}{c}\ba \\\bb\end{array}\right]\right\|_2^2,
\end{equation}  
where $W=\diag(\bw)$ and $F=\diag(\Bf)$. It is known that the minimization of \eqref{eq:rat-d-compt} is reachable at a pair $(\wtd \ba,\wtd \bb)$  as  it is  a trace minimization for a Hermitian positive semi-definite pencil, and  the solution can be obtained by  \cite[Theorem 2.1]{lilb:2013} (more detailed information on computation of $(\wtd \ba,\wtd \bb)$ will be given in section \ref{sec_d2}).  Let $(\wtd p,\wtd q)$ be the corresponding   polynomials  for \eqref{eq:rat-d}. 

As $\bw\ne 0$ and $q^*(x_j)\ne 0$, we can always choose a scaling $\tau$ so that $p_\tau^*= \tau p^*, ~q_\tau^*= \tau q^* $ satisfying $\sum_{j=1}^m  w_j|q_\tau^*(x_j)|^2 =1$. Thus
\begin{align}\nonumber
d_2(\bw)&=\sum_{j=1}^m w_j |f_j \wtd q(x_j)-\wtd  p(x_j)|^2 \le\sum_{j=1}^m w_j |f_j  q_\tau^*(x_j)-  p_\tau^*(x_j)|^2 \\\label{eq:weakduality}
&= \sum_{j=1}^m w_j| q_\tau^*(x_j)|^2\cdot \left|f_j-\frac{ p_\tau^*(x_j)}{ q_\tau^*(x_j)}\right|^2 
 \le \sum_{j=1}^m w_j| q_\tau^*(x_j)|^2\cdot \|\Bf-\xi^*\|_\infty^2= (\eta_\infty)^2.
\end{align}

Finally, to show $$\sup_{\bw\ge \mathbf{0}}d_2(\bw) =\max_{\bw\in {\cal S}}d_2(\bw),$$ we only need to show 
$\sup_{\bw\ge \mathbf{0}}d_2(\bw) \le \max_{\bw\in {\cal S}}d_2(\bw)$. For any $\bw\ge \mathbf{0}$, let $\wtd \bw= {\bw}{\tau}\in {\cal S}$ with $\tau=\frac{1}{\bw^{\T}\be}$. Note 
\begin{align*}
d_2(\wtd \bw)&=\min_{\begin{subarray}{c}  \ba\in \bbC^{n_1+1},~ \bb\in \bbC^{n_2+1}\\
            \|\sqrt{\wtd W}\Phi \bb\|_2 =1\end{subarray}}\left\|\sqrt{\wtd W}[-\Psi,F\Phi]  \left[\begin{array}{c}\ba \\\bb\end{array}\right]\right\|_2^2\\
            &= \min_{\begin{subarray}{c}  \sqrt{\tau}\ba\in \bbC^{n_1+1},~ \sqrt{\tau} \bb\in \bbC^{n_2+1}\\
            \|\sqrt{W}\Phi(\sqrt{\tau}\bb)\|_2 =1\end{subarray}}\left\|\sqrt{  W}[-\Psi,F\Phi]  \left[\begin{array}{c}\sqrt{\tau}\ba \\\sqrt{\tau} \bb\end{array}\right]\right\|_2^2=d_2( \bw),
\end{align*}
and thus the proof is complete.
\end{proof}
  
\begin{remark}\label{rk:dual}
\item[(1)] Note that the dual {problem} $\max_{\bw\in {\cal S}} d_2(\bw)$ is a convex optimization {problem} (i.e., maximize a concave function over a convex set; see e.g., \cite[Chapter 5.2]{boyd:2004}). 
A further desired property is the so-called {\it strong duality}, that is,
\begin{equation}\label{eq:strongdual}
d_2^*:=\max_{\bw\in {\cal S}} d_2(\bw)=(\eta_\infty)^2
\end{equation} 
which implies that there is a  $\bw^*\in {\cal S}$ so that $d_2(\bw^*)=(\eta_\infty)^2$. For the special case when $\eta_\infty=0$, the solution $\xi^*=p^*/q^*$ is an interpolation for $\Bf$ satisfying $\xi^*(x_j)=f_j$; in this case, strong duality holds because   for any $\bw\in {\cal S}$, $d_2(\bw)=0$ due to $[-\Psi,F\Phi]\left[\begin{array}{c} \ba^* \\ \bb^*\end{array}\right]=0$ where $\ba^*\in \bbC^{n_1+1}$, $\bb^*\in \bbC^{n_2+1}$ are the corresponding coefficient vectors of $p^*$ and $q^*$, respectively.

\item[(2)] Another special case is for $n_2=0$. In this case, $q\in \bbP_0$ and thus we have $q(x) \equiv  b_1$ for some $b_1\in \bbC$. Now, for any $\bw\in {\cal S}$, the constraint $\sum_{j=1}^m w_j |q(x_j)|^2=\sum_{j=1}^m w_j |b_1|^2=1$ leads to $|b_1|=1$ and thus  $d_2(\bw)$ in \eqref{eq:rat-d} becomes
\begin{equation}\nonumber
d_2(\bw)=\min_{\begin{subarray}{c}  \ba\in \bbC^{n_1+1} \end{subarray}}\left\|\sqrt{W}(\Bf b_1-\Psi \ba ) \right\|_2^2 \xlongequal{\text{$\wtd \ba = \ba/b_1$}} \min_{\begin{subarray}{c}  \wtd\ba\in \bbC^{n_1+1} \end{subarray}}\left\|\sqrt{W}(\Bf -\Psi \wtd \ba )  \right\|_2^2
\end{equation}  
which is the exact $L_2$-weighted dual function for the linear minimax problem satisfying   strong duality: $\max_{\bw\in {\cal S}}d_2(\bw)=(\eta_\infty)^2$ (see \cite[Theorem 2.1]{yazz:2023}).

\item[(3)] In general, in section \ref{sec_strongduality}, we will  connect   strong duality \eqref{eq:strongdual}   with Ruttan's sufficient condition \eqref{eq:Ruttan} \cite[Theorem 2.1]{rutt:1985} (see also \cite[Theorem 3]{this:1993}) and the local  Kolmogorov necessary condition (see e.g., \cite[Theorem 1.2]{rutt:1985},  \cite[Theorem 1]{isth:1993} and \cite[Theorem 2]{this:1993})  for a global best rational   approximation of \eqref{eq:bestf}. In particular, we shall show in Theorem \ref{thm:strongdualityeqvRuttan} that   strong duality is equivalent to Ruttan's sufficient condition \eqref{eq:Ruttan}. Moreover, we shall report numerical results to demonstrate that   strong duality (i.e., Ruttan's sufficient condition \eqref{eq:Ruttan}) often holds in practice.
\end{remark}
It is also interesting to point out that the term $\sum_{j=1}^m w_j |f_j q(x_j)-p(x_j)|^2$ has already been used in the rational approximations (see e.g.,  \cite{begu:2015,begu:2017,gogu:2021}) where $\{w_j\}_{j=1}^m$ are interpreted as weights  in the least-squares mode. It is our main goal to develop updating rules for these dual variables and thus lead to various algorithms for solving \eqref{eq:bestf}.

The following result  states that the number of {\it reference points} ({aka} the extreme points), i.e., nodes $x_j\in {\cal X}$ that {achieve} $|f_j- p^*(x_j)/q^*(x_j)|=e(\xi^*)$,  for a solution $p^*/q^*$ of \eqref{eq:bestf} is at least $n_1+n_2+2-\upsilon(p^*,q^*)$. 
\begin{theorem}{\rm (\cite[Theorem 2.5]{gutk:1983})}\label{thm:extremalNo}
Given $m\ge {n_1}+n_2+2$ distinct nodes ${\cal X}=\{x_j\}_{j=1}^m$ on $\Omega$, suppose $\xi^*=p^*/q^*\in \scrR_{(n_1,n_2)}$ is an irreducible rational {approximant} and denote the extremal set   $ {\cal  X}_e(\xi^*)\subseteq {\cal X}$ by 
\begin{equation}\label{eq:extremalset}
 {\cal  X}_e(\xi^*):=\left\{x_j\in {\cal X}:\left|f_j-\frac{p^*(x_j)}{q^*(x_j)}\right|=e(\xi^*)\right\}. 
\end{equation}
If $\xi^*$ is a solution to \eqref{eq:bestf} with 
$
\eta_\infty=\|\Bf-\xi^*(\bx)\|_\infty,
$
 then the cardinality  $|{\cal  X}_e(\xi^*)|\ge  n_1+n_2+2-\upsilon(p^*,q^*)$; that is,  ${\cal  X}_e(\xi^*)$ contains at least $n_1+n_2+2-\upsilon(p^*,q^*)$ nodes. 
\end{theorem}

In the following theorem, we show that when   strong duality \eqref{eq:strongdual} holds,  one can choose to  filter out non-reference points in ${\cal X}$ to reduce the computational costs and accelerate the convergence for solving the dual {problem} \eqref{eq:dualPX}. The same strategy has been used in Lawson's iteration for the linear minimax approximations \cite{clin:1972,laws:1961,rice:1969,yazz:2023}.

\begin{theorem}\label{thm:complement}
Under {the assumptions of} Theorem \ref{thm:extremalNo}, if   strong duality \eqref{eq:strongdual} holds, then  
\begin{itemize}
\item[(i)] {\rm (complementary slackness)} for any solution $\bw^*$ of the   dual {problem}:
\begin{equation}\label{eq:dualPX}
\max_{\bw\in {\cal S}} d_2(\bw),
\end{equation}
{where $d_2(\bw)$ is given in \eqref{eq:rat-d-compt}}, it holds that 
$   w_j^*(\eta_\infty -|f_j-\xi^*(x_j)|)=0, \quad\forall j=1,2,\dots,m$;
\item[(ii)]  {\rm (filtering out non-reference points)}  for any subset $\breve{\cal X} \subseteq{\cal X}$ satisfying ${\cal X}_e(\xi^*)\subseteq \breve{\cal X}$, the triple   $((\eta_\infty)^2, p^*,q^*)$ solves 
 \begin{align}\nonumber
&\inf_{\eta\in \bbR,~p\in \bbP_{n_1},~q\in \bbP_{n_2}\setminus\{0\}}\eta \\\label{eq:linearitysubset}
 s.t., ~& |f(x_j)q(x_j)-p(x_j)|^2\le \eta |q(x_j)|^2,~~\forall x_j\in \breve{\cal X},
\end{align}
for which   strong duality  holds too; that is, 
\begin{equation}\label{eq:dualPsubX}
(\eta_\infty)^2 =\max_{\breve{\bw}\in \breve{{\cal S}}} \breve{d}_2(\breve{\bw}),
\end{equation}  
where $\breve{{\cal S}}=\{\breve{\bw}=[\breve{w}_1,\dots,\breve{w}_{\breve{s}}]^{\T}\in \bbR^{\breve{s}}: \breve{\bw}\ge 0 ~{\rm and } ~\breve{\bw}^{\T}\be=1\},
$ $\breve{s}=|\breve{{\cal X}}|$, and
\begin{equation}\label{eq:rat-d-subset}
\breve{d}_2(\breve{\bw})=\min_{\begin{subarray}{c} p\in \bbP_{n_1},~q\in \bbP_{n_2}\\
            \sum_{x_j\in\breve{\cal X}} \breve{w}_j |q(x_j)|^2=1\end{subarray}}~\sum_{x_j\in\breve{\cal X}} \breve{w}_j |f(x_j) q(x_j)-p(x_j)|^2.
\end{equation}  
Moreover, for any solution $\breve{\bw}^*$ of the dual \eqref{eq:dualPsubX}, the pair $(p^*,q^*)$ achieves the minimum  $\breve{d}_2(\breve{\bw}^*)=(\eta_\infty)^2$ of \eqref{eq:rat-d-subset}.
\end{itemize}
\end{theorem}
\begin{proof}
For (i), by   strong duality \eqref{eq:strongdual} and since $(p^*,q^*)$ is feasible (by a scaling for $q^*$ using $\bw^*$) for the minimization \eqref{eq:rat-d} at $\bw=\bw^*$, we have 
\begin{align*}
(\eta_\infty)^2=d_2(\bw^*)& \le \sum_{j=1}^m w^*_j|f(x_j)q^*(x_j)-p^*(x_j)|^2 \\
&=  \sum_{j=1}^m w^*_j|q^*(x_j)|^2 |f(x_j)-\xi^*(x_j)|^2\\
&\le  (\eta_\infty)^2\sum_{j=1}^m w^*_j|q^*(x_j)|^2=(\eta_\infty)^2,
\end{align*}
which yields $$w^*_j|q^*(x_j)|^2 \left((\eta_\infty)^2-|f(x_j)-\xi^*(x_j)|^2\right)=0, \forall j\in [m].$$ As $q^*$ has no pole at ${\cal X}$, the result $w_j^*(\eta_\infty -|f(x_j)-\xi^*(x_j)|)=0$ $\forall j\in[m]$  follows.

For (ii), according to Theorem \ref{thm:extremalNo} and (i), we know for any solution $\bw^*=[w_1^*,\dots,w_m^*]^{\T}$ of the dual \eqref{eq:dualPX}, $w_j^*=0$ for any $j$ with $x_j\not\in {\cal X}_e(\xi^*)$ in \eqref{eq:extremalset}. For simplicity, assume $ {\cal X}_e(\xi^*)=\{x_j\}_{j=1}^{s}$ and $\breve{\cal X}=\{x_j\}_{j=1}^{\breve{s}}$ with $\breve{s}\ge s$. Denote by $\breve{\eta}_\infty$ the infimum of \eqref{eq:linearitysubset} and by $\breve{d}_2^*$ the maximum of \eqref{eq:dualPsubX}. Similarly to \eqref{eq:rat-dual}, we have the weak duality $\breve{d}_2^*\le (\breve{\eta}_\infty)^2$. Also, as $((\eta_\infty)^2, p^*,q^*)$ is feasible for \eqref{eq:linearitysubset}, $\breve{\eta}_\infty\le {\eta}_\infty$; thus 
\begin{equation}\label{eq:subdual1}
\breve{d}_2^*=\breve{d}_2(\breve{\bw}^*)\le (\breve{\eta}_\infty)^2\le ({\eta}_\infty)^2=d_2(\bw^*)
\end{equation} 
where $\breve{\bw}^*$ and  $ {\bw}^*$ are solutions  for  the duals \eqref{eq:dualPsubX} and  \eqref{eq:dualPX}, respectively.  On the other hand, note, by the complementary slackness, that $w_j^*=0$ for any $j\ge s$, and hence
\begin{align*}
{d}_2({\bw}^*)&=\min_{\begin{subarray}{c}p\in \bbP_{n_1},~q\in \bbP_{n_2}\\
            \sum_{j=1}^m  {w}_j^* |q(x_j)|^2=1\end{subarray}}\sum_{j=1}^m  {w}_j^* |f(x_j) q(x_j)-p(x_j)|^2\\
            &=\min_{\begin{subarray}{c} p\in \bbP_{n_1},~q\in \bbP_{n_2}\\
            \sum_{j=1}^{\breve{s}}  {w}_j^* |q(x_j)|^2=1\end{subarray}}\sum_{j=1}^{\breve{s}}  {w}_j^* |f(x_j) q(x_j)-p(x_j)|^2 = \breve{d}_2({\bw}^*_{[1:\breve{s}]}) \le \breve{d}_2^*,
\end{align*}
where the last inequality follows because $\breve{d}_2^*$ is the maximum of \eqref{eq:dualPsubX} over all $\breve\bw\in \breve{\cal S}$. Together with \eqref{eq:subdual1}, we conclude that $\breve{d}_2^*=\breve{d}_2(\breve{\bw}^*)= (\breve{\eta}_\infty)^2= ({\eta}_\infty)^2=d_2(\bw^*)$. Finally, for any solution $\breve{\bw}^*$ of the dual \eqref{eq:dualPsubX}, suppose the  pair $(\hat p,\hat q)$ achieves the minimum  $\breve{d}_2(\breve{\bw}^*)=(\eta_\infty)^2$ of \eqref{eq:rat-d-subset}. Scale $(p^*,q^*)$ to satisfy the constraint $\sum_{j=1}^{\breve{s}}  \breve{w}_j^* |q^*(x_j)|^2=1$ (i.e., $(p^*,q^*)$ is feasible for the minimization \eqref{eq:rat-d-subset} with $\breve{\bw}=\breve{\bw}^*$), and thus by 
\begin{align*}
\breve{d}_2^*=\breve{d}_2(\breve{\bw}^*)& = \sum_{j=1}^{\breve{s}} \breve{w}^*_j|f(x_j)\hat q(x_j)-\hat p(x_j)|^2\le  \sum_{j=1}^{\breve{s}} \breve{w}^*_j|f(x_j) q^*(x_j)- p^*(x_j)|^2\\
&=  \sum_{j=1}^{\breve{s}} \breve{w}^*_j|q^*(x_j)|^2 |f(x_j)-\xi^*(x_j)|^2  \le  (\eta_\infty)^2\sum_{j=1}^{\breve{s}} \breve{w}^*_j|q^*(x_j)|^2=(\eta_\infty)^2=\breve{d}_2^*,
\end{align*}
we know $(p^*,q^*)$ achieves the minimum  $\breve{d}_2(\breve{\bw}^*)=(\eta_\infty)^2$ of \eqref{eq:rat-d-subset}.
\end{proof}

\section{Computation of $d_2(\bw)$}\label{sec_d2} 

\subsection{Optimality for the dual problem and computation of $d_2(\bw)$}
{We next show that computing the dual function value $d_2(\bw)$ at $\bw\in {\cal S}$, i.e., the minimum of \eqref{eq:rat-d}, can be cast as a generalized eigenvalue problem.}

\begin{proposition}\label{prop:dual_GEP}
For   $ \bw\in {\cal S}$, we have  
\begin{itemize}
\item[(i)]
$\bc(\bw)=\left[\begin{array}{c} \ba(\bw) \\ \bb(\bw)\end{array}\right] \in \bbC^{n_1+n_2+2}$ is a solution of \eqref{eq:rat-d} if and only if it {is} an eigenvector of the Hermitian positive semi-definite generalized eigenvalue problem $(A_{\bw},B_{\bw})$ and  $d_2(\bw)$ is the smallest eigenvalue  satisfying
\begin{equation}\label{eq:dual_GEP}
A_{\bw} \bc(\bw)=d_2(\bw) B_{\bw} \bc(\bw)~ \mbox{ and }~ \bc(\bw)^{\HH}B_{\bw}\bc(\bw) =1,
\end{equation} 
where
\begin{align}\label{eq:dual_GEPA}
A_{\bw}:&=[-\Psi,F\Phi]^{\HH}W[-\Psi,F\Phi]=\left[\begin{array}{cc}\Psi^{\HH}W\Psi & -\Psi^{\HH}FW\Phi \\-\Phi^{\HH} WF^{\HH}\Psi & \Phi^{\HH}F^{\HH}WF\Phi\end{array}\right],\\\label{eq:dual_GEPB}
B_{\bw}:&= [0,\Phi]^{\HH}W [0,\Phi]=\left[\begin{array}{cc}0 & 0 \\0 & \Phi^{\HH}W\Phi \end{array}\right];
\end{align}
\item[(ii)] the  Hermitian matrix $H_{\bw}:=A_{\bw} -d_2(\bw) B_{\bw} \succeq 0$, i.e., $H_{\bw}$ is  positive semi-definite;

\item[(iii)] let $W^{\frac12}\Phi=Q_qR_q$ and $W^{\frac12}\Psi=Q_pR_p$ be the thin QR factorizations where $Q_q\in \bbC^{m\times \wtd n_2}$, $Q_p\in \bbC^{m\times \wtd n_1}$, $R_q\in \bbC^{ \wtd n_2 \times (n_2+1)}$,  $R_p\in \bbC^{\wtd n_1 \times (n_1+1)}$ with $\wtd n_1=\rank(W^{\frac12}\Psi)$ and $\wtd n_2=\rank(W^{\frac12}\Phi)$. Then $d_2(\bw)$ is the smallest   eigenvalue of the following Hermitian positive semi-definite  matrix 
\begin{equation}\label{eq:Heig}
S(\bw):=S_F-S_{qp}S_{qp}^{\HH}\in \bbC^{\wtd n_2\times \wtd n_2},
\end{equation} 
where 
\begin{equation}\nonumber
S_F=Q_q^{\HH} |F|^2Q_q\in \bbC^{\wtd n_2\times \wtd n_2},~~S_{qp}=Q_q^{\HH}F^{\HH} Q_p\in \bbC^{\wtd n_2 \times \wtd n_1}.
\end{equation}
Moreover, let $[Q_p,Q_p^{\perp}]\in \bbC^{m\times m}$ be unitary, then $\sqrt{d_2(\bw)}$ is also the smallest singular value of both $(Q_p^{\perp})^{\HH}FQ_q\in \bbC^{(m-\wtd n_1-1)\times (\wtd n_2+1)}$ {and $(I-Q_pQ_p^{\HH})FQ_q\in \bbC^{m\times (\wtd n_2+1)}$.}
\end{itemize}
\end{proposition}

\begin{proof}
The result of (i) can be obtained from, e.g., 
\cite[Theorem 8.7.1]{govl:2013} or \cite{lilb:2013}; (ii) can be seen from optimality. Indeed, if there is $\wtd \bc$ so that $ \wtd \bc ^{\HH}H_{\bw}\wtd \bc<0$, by the positive semi-definiteness of $A_{\bw}$ and $B_{\bw}$, it holds {that} $\wtd\bc ^{\HH}B_{\bw}\wtd \bc>0$. Normalize $\wtd \bc$ so that $\wtd\bc ^{\HH}B_{\bw}\wtd \bc=1$ to conclude $\wtd\bc ^{\HH}A_{\bw}\wtd \bc<d_2(\bw)$, a contradiction with the fact that $d_2(\bw)$ achieves the minimum.

For (iii), let 
\begin{equation}\nonumber
\what \ba = R_p\ba(\bw),~\what \bb =R_q\bb(\bw)
\end{equation}
and write \eqref{eq:dual_GEP}  as 
$$
R_p^{\HH}(\what \ba-S_{qp}^{\HH}\what \bb) =0,~R_q^{\HH}(-S_{qp}\what \ba+S_{F}\what \bb-d_2(\bw) \what\bb)=0,~~ \|\what \bb\|_2=1.
$$ 
As $R_p$ and $R_q$ are both of full row rank, we have  
$$
(S_{F}-S_{qp}S_{qp}^{\HH})\what\bb=d_2(\bw) \what \bb,~~\what \ba=S_{qp}^{\HH}\what \bb,~~\|\what \bb\|_2=1.
$$
As $$S_{F}-S_{qp}S_{qp}^{\HH}=Q_q^{\HH}F^{\HH}(I_{m} -Q_pQ_p^{\HH})F Q_q,$$ the 
 positive semi-definiteness of $S_{F}-S_{qp}S_{qp}^{\HH}$  follows. Furthermore, as  
 $$
 S_{F}-S_{qp}S_{qp}^{\HH}=Q_q^{\HH}F^{\HH}(I_{m} -Q_pQ_p^{\HH})FQ_q={Q_q^{\HH}F^{\HH}(I_{m} -Q_pQ_p^{\HH})^2FQ_q=Q_q^{\HH}F^{\HH} Q_p^{\perp}(Q_p^{\perp})^{\HH}F Q_q,}
 $$
 $\sqrt{d_2(\bw)}$ is also the smallest singular value of both $(Q_p^{\perp})^{\HH}FQ_q$ and  {$(I_{m} -Q_pQ_p^{\HH})FQ_q$,} and $\what \bb$ is the corresponding right singular vector. 
 \end{proof}

\begin{remark}\label{rmk:posw}
{Note that the normalized condition $\bc(\bw)^{\HH}B_{\bw}\bc(\bw) =1$ in \eqref{eq:dual_GEP} imposes an additional requirement on the eigenvector $\bc(\bw)$ of $(A_{\bw},B_{\bw})$. We remark that, if $q \not \equiv 0$,  $\bc(\bw)^{\HH}B_{\bw}\bc(\bw) $ can always be normalized to be 1 whenever there are at least $n_2+1$ positive elements in $\bw$. To see this,  if  $0=\bc(\bw)^{\HH}B_{\bw}\bc(\bw) =\sum_{j=1}^mw_j|q(x_j)|^2,$ then we have  $w_j q(x_j)=0$ for any $j\in [m]$; therefore, whenever there are at least $n_2+1$ positive elements  $w_j$, the associated nodes $x_j$ are the zeros of $q$, which then leads to $q  \equiv 0$. Due to this fact, in the following, we assume, without loss of generality, that for any $\bw^{(k)}$  obtained during an iteration, it has at least $n_2+1$ positive elements.
}
\end{remark}

\subsection{More techniques for computing  $d_2(\bw)$}\label{subsec:d2accurate}

It is noted that for computing $d_2(\bw)$, we only need the orthogonal factors $Q_p$ and $Q_q$; these two matrices  can be obtained via the Arnoldi process  \cite{brnt:2021,hoka:2020,zhsl:2023}   without involving explicitly the Vandermonde matrix $V_{\bx,n_1}=[\be,X\be,\dots,X^{n_1}\be]$ with $X={\rm diag}(\bx)$. In particular, noticing  
$$
\sqrt{W}V_{\bx,n_1}=[\sqrt{\bw},X\sqrt{\bw},\dots,X^{n_1}\sqrt{\bw}],
$$
we know that $Q_p$ {with $Q_p\be_1=\sqrt{\bw}/\|\sqrt{\bw}\|_2$} is the basis for the Krylov subspace ${\cal K}_{n_1+1}(\sqrt{\bw}, X)$ (see \cite[Theorem 2.1]{zhsl:2023} and also \cite{hoka:2020}). The same applies to $\sqrt{W}V_{\bx,n_2}$ to yield $Q_q$ {with $Q_q\be_1=\sqrt{\bw}/\|\sqrt{\bw}\|_2$}.

In practice, we recommend to compute $\sqrt{d_2(\bw^{(k)})}$ as the smallest singular  of $(Q_p^{\perp})^{\HH}FQ_q$ for small- to medium-size $m$.  One the other hand, calling of the conventional SVD is backward stable; this implies that the relative accuracy of the numerical value $\sqrt{\hat{d_2}(\bw)}$ of $\sqrt{d_2(\bw)}$ is dependent on the condition number $\kappa_2((Q_p^{\perp})^{\HH}FQ_q)$ of  $(Q_p^{\perp})^{\HH}FQ_q$ and  satisfies (see \cite[Equation (3)]{dges:1999})
$$
\frac{\left|\sqrt{\hat{d_2}(\bw)}-\sqrt{{d_2}(\bw)}\right|}{ \sqrt{{d_2}(\bw)}}\le  O( \tu) \cdot \kappa_2((Q_p^{\perp})^{\HH}FQ_q),
$$  
where $\tu$ is the   machine precision. Fortunately, as both $(Q_p^{\perp})^{\HH}$ and $Q_q$ are well conditioned and $F$ is diagonal, the technique of \cite{dges:1999} and   Jacobi's method \cite{demm:1992} can be used so that {the} computed singular value $\sqrt{\hat{d_2}(\bw)}$ is guaranteed to have high relative accuracy even when $(Q_p^{\perp})^{\HH}FQ_q$ is ill-conditioned. {For large $m$, we suggest to compute $\sqrt{d_2(\bw^{(k)})}$ as the smallest singular value  of $(I_{m} -Q_pQ_p^{\HH})FQ_q$.}

Inspired by the technique of Vandermonde with Arnoldi \cite{brnt:2021}, another noticeable point that can avoid the ill-conditioning {of the} Vandermonde matrices (i.e., avoid explicitly involving  the triangular matrices $R_p$ and $R_q$) during  {solving of the dual problem} \eqref{eq:dualPX} is to use the vectors $\what \ba$ and $\what \bb$, instead of $\ba$ and $\bb$.  Specifically,   we can compute  vectors
\begin{subequations}\label{eq:pqvector}
\begin{align}\label{eq:pqvectora}
\bp &=V_{\bx,n_1}\ba(\bw)=W^{-\frac12}Q_pR_p\ba(\bw)=W^{-\frac12}Q_p\what\ba,\\
\bq&=V_{\bx,n_2}\bb(\bw) =W^{-\frac12}Q_qR_q\bb(\bw)=W^{-\frac12}Q_q\what\bb,\label{eq:pqvectorb}
\end{align}
\end{subequations} 
where $\what \bb$ is the right singular vector of both $(Q_p^{\perp})^{\HH}FQ_q$ and $(I_{m} -Q_pQ_p^{\HH})FQ_q$ associated with the smallest singular value, and   $\what \ba=S_{qp}^{\HH}\what \bb$. The same technique has been applied for dealing {with} the linear minimax approximation from a similar dual {formulation} \cite{zhsl:2023}.

\subsection{Evaluation  at new nodes}
{According to \eqref{eq:pqvector}, at the maximizer $\bw^*\in {\cal S}$ of the dual problem \eqref{eq:dualPX},} suppose that   the corresponding vectors 
\begin{align*}
p^*(\bx)&=\bp^* =V_{\bx,n_1}\ba(\bw^*)=(W^*)^{-\frac12}Q_p^*(R_p^*\ba(\bw^*))=(W^*)^{-\frac12}Q_p^*\what\ba^*,\\
q^*(\bx)&=\bq^*=V_{\bx,n_2}\bb(\bw^*)=(W^*)^{-\frac12}Q_q^*(R_q^*\bb(\bw^*))= (W^*)^{-\frac12}Q_q^*\what\bb^*,
\end{align*} 
of the numerator and denominator of $\xi^*(x)=p^*(x)/q^*(x)$ at the sample vector $\bx$ have been computed; the technique introduced in \cite{brnt:2021} then facilitates computing the values of $\{\xi^*(y_j)\}_{j=1}^{\wtd m}$ at new nodes $\{y_j\}_{j=1}^{\wtd m}$. Specifically, let $\by=[y_1,\dots,y_{{\wtd m}}]^{\T}\in \bbC^{\wtd m}$; then the Arnoldi process for generating the orthogonal matrices $Q_p^*$ and $Q_q^*$ can analogously be used with the new sample vector $\by$ to get the associated  (not necessarily orthonormal) basis matrices $L_p$ and $L_q$ (see \cite[Theorem 2.1, Figure 2.1,  Equations (2.8) and (2.9)]{zhsl:2023}) for $V_{\by,n_1}$ and $V_{\by,n_2}$, respectively; it holds that {$L_p\be_1 = \be\in \bbR^{n_1+1}$, $L_q\be_1 = \be\in \bbR^{n_2+1}$,  and 
$ 
V_{\by,n_1}=L_p R_p^*,~~V_{\by,n_2}=L_q R_q^* ;
$ 
 consequently, 
 \begin{subequations}\label{eq:new_pqvector}
 \begin{align}\label{eq:new_pqvectora}
p^*(\by)&=V_{\by,n_1}\ba(\bw^*)= L_p^*(R_p^*\ba(\bw^*))= L_p^*\what\ba^*,\\\label{eq:new_pqvectorb}
q^*(\by)&=V_{\by,n_2}\bb(\bw^*)= L_q^*(R_q^*\bb(\bw^*))=  L_q^*\what\bb^*,
\end{align} 
\end{subequations}
  and $\xi^*(\by)=(p^*(\by))./q^*(\by)=(L_p\what \ba^*)./(L_q\what \bb^*)$.
  }

\section{Strong duality  and Ruttan's sufficient condition}\label{sec_strongduality} 
In the literature, necessary/sufficient optimality conditions of a local best rational  approximation of \eqref{eq:bestf} and   uniqueness have been established in e.g., \cite{elli:1978,gutk:1983,rutt:1985,sava:1977,this:1993,wulb:1980}. In particular, \cite{rutt:1985} contributes both a necessary, the Kolmogorov local condition (dual form), for a local best   rational approximation, and a sufficient condition for a global best of \eqref{eq:bestf}. In this section, we shall see that these two optimality conditions are closely related with   strong duality \eqref{eq:strongdual}.

To state Ruttan's  {sufficient optimality} condition, for an irreducible $\xi=p/q\in\scrR_{(n_1,n_2)}$ and  a given $x\in \bbC$, we introduce a matrix defined by
\begin{equation} \label{eq:Rutten-rank1}
H(x)=\left[\begin{array}{cc}H_1(x) & H_3(x)  \\H_3^{\HH}(x)  &H_2(x) \end{array}\right]\in \bbC^{(n_1+n_2+2)\times (n_1+n_2+2)}
\end{equation}
where (with $e(\xi)$ given in \eqref{eq:exi})
\begin{subequations}\label{eq:Hk}
\begin{align}\label{eq:H1}
[H_1(x)]_{\ell,k}&=\overline{\psi_{\ell-1}(x)}\psi_{k-1}(x)\in \bbC,~~1\le \ell,k\le n_1+1,\\\label{eq:H2}
[H_2(x)]_{\ell,k}&=\left(|f(x)|^2- (e(\xi))^2\right)\overline{\phi_{\ell-1}(x)}\phi_{k-1}(x)\in \bbC,~~1\le \ell,k\le n_2+1,\\\label{eq:H3}
[H_3(x)]_{\ell,k}&=-f(x)\overline{\psi_{\ell-1}(x)}\phi_{k-1}(x)\in \bbC,~~1\le \ell\le n_1+1,~~1\le k\le n_2+1.
\end{align}
\end{subequations}
Then Ruttan's sufficient condition (\cite[{Theorem} 2.1]{rutt:1985};  see also \cite[Theorem 2]{isth:1993} and \cite[Theorem 3]{this:1993}) for a global best approximant $\xi^*=p^*/q^*$ can be summarized in the following theorem.

\begin{theorem}{\rm (\cite[{Theorem} 2.1]{rutt:1985})}\label{thm:RuttanSufOpt}
An irreducible rational {approximant}  $\hat\xi=\hat p /\hat q \in  \scrR_{(n_1,n_2)}$ is a best approximation for \eqref{eq:bestf} if there exist points $\{z_j\}_{j=1}^t\subseteq {\cal X}_e(\hat\xi),~t\ge 1$ and positive real constants $\{\varpi_j\}_{j=1}^t$ such that $\sum_{j=1}^t \varpi_j=1$ and   the Hermitian matrix 
\begin{equation}\label{eq:RuttanHc}
H =\sum_{j=1}^t \varpi_j H(z_j)\succeq 0.
\end{equation}
In that case the points $\{z_j\}_{j=1}^t$ and $\{\varpi_j\}_{j=1}^t$ satisfy the local (dual form) Kolmogorov condition 
\begin{equation}\nonumber
\sum_{j=1}^t \varpi_j \overline{(f(z_j)-\hat\xi(z_j))  \hat q(z_j)} p(z_j)/\hat q (z_j)=0,\quad \forall p\in \bbP_{N},
\end{equation}
where $N=n_1+n_2-\upsilon(\hat p,\hat q)$ and $\upsilon(\hat p,\hat q)$ is the defect of $\hat\xi$ given by \eqref{eq:defect}.
\end{theorem}

 In the following,  Ruttan's sufficient condition is stated as 
\begin{equation}\label{eq:Ruttan}
\hspace{2mm}
\framebox{\parbox{13cm}{
there is an irreducible (not necessarily non-degenerate)  $\hat\xi=\hat p /\hat q \in  \scrR_{(n_1,n_2)}$, a subset $\{z_j\}_{j=1}^t\subseteq {\cal X}_e(\hat\xi),~t\ge 1$ and positive real constants $\{\varpi_j\}_{j=1}^t$ with  $\sum_{j=1}^t \varpi_j=1$ so that   
$ H =\sum_{j=1}^t \varpi_j H(z_j)\succeq 0$, {where $H(z_j)$ is defined in \eqref{eq:Rutten-rank1}.}}
}
\end{equation}

{We remark that Ruttan's sufficient condition \eqref{eq:Ruttan} alone is not very practical in checking a given $\hat\xi=\hat p /\hat q \in  \scrR_{(n_1,n_2)}$ to be minimax approximant. Indeed, one first needs to compute the set of reference points ${\cal X}_e(\hat\xi)$, and then choose a suitable subset  $\{z_j\}_{j=1}^t\subseteq{\cal X}_e(\hat\xi)$ together with the corresponding $\{\varpi_j\}_{j=1}^t$ so that $ H =\sum_{j=1}^t \varpi_j H(z_j)\succeq 0.$ Interestingly, we shall see that Ruttan's sufficient condition \eqref{eq:Ruttan} perfectly fits in with the dual problem \eqref{eq:dualPX}:  a rational approximant will be computed efficiently from \eqref{eq:dualPX} for which Ruttan's sufficient condition \eqref{eq:Ruttan} can be checked easily.}

\begin{theorem}\label{thm:strongdualityRuttan}
Let $\bw^*\in{\cal S}$ be the maximizer of the convex dual {problem} \eqref{eq:dualPX}. 
Ruttan's sufficient condition \eqref{eq:Ruttan} is satisfied  if and only if there is an irreducible rational {approximant}  $\hat \xi =\hat  p /\hat  q \in  \scrR_{(n_1,n_2)}$ {that} achieves the minimum $d_2(\bw^*)$  of \eqref{eq:rat-d-compt} with $\bw=\bw^*$ and 
\begin{equation}\label{eq:strongdualityRuttan}
\sqrt{d_2(\bw^*)}=e(\hat  \xi).
\end{equation}
\end{theorem}
\begin{proof}
For the sufficiency, suppose the irreducible rational {approximant} $(\hat  p,\hat  q)$ achieves the minimum $d_2(\bw^*)$  of \eqref{eq:rat-d-compt} with $\bw=\bw^*$ and $d_2(\bw^*)=(e(\hat  \xi))^2$. {Performing}  the same argument as that for \eqref{eq:weakduality}, $d_2(\bw^*)=(e(\hat  \xi))^2$ leads to the  complementary condition: $w_j^*=0$ for any $x_j\not \in {\cal X}_e(\hat \xi)$  given by \eqref{eq:extremalset}. 
 By the optimality in Proposition \ref{prop:dual_GEP} (ii), we also know that the Hermitian matrix $H_{\bw^*}= A_{\bw^*} -d_2(\bw^*) B_{\bw^*}$ is  positive semi-definite. Notice from \eqref{eq:dual_GEPA} and \eqref{eq:dual_GEPB} that 
\begin{align*}
 H_{\bw^*}&= A_{\bw^*} -d_2(\bw^*) B_{\bw^*}= \left[\begin{array}{cc}\Psi^{\HH}W\Psi & -\Psi^{\HH}FW\Phi \\-\Phi^{\HH} WF^{\HH}\Psi & \Phi^{\HH}F^{\HH}WF\Phi-d_2(\bw^*)\Phi^{\HH}W\Phi\end{array}\right],\end{align*}
 and by the definitions of $H_k(x)$ for $k=1,2,3$ in \eqref{eq:Hk},  it holds that $$[\Psi^{\HH}\be_j\be_j^{\T}\Psi]_{\ell,k}=\be_{\ell}^{\T}\Psi^{\HH}\be_j\be_j^{\T}\Psi\be_k=\overline{\psi_{\ell-1}(x_j)}\psi_{k-1}(x_j),$$ and thus
 \begin{subequations}\label{eq:HwRuttan}
 \begin{align}\label{eq:HwRuttan1}
 &\Psi^{\HH}W\Psi=\sum_{j=1}^m w_j^* \Psi^{\HH}\be_j\be_j^{\T}\Psi=\sum_{j=1}^m w_j^*  H_1(x_j),~ -\Psi^{\HH}FW\Phi=\sum_{j=1}^m w_j^*  H_3(x_j),\\\label{eq:HwRuttan2}
 &\Phi^{\HH}F^{\HH}WF\Phi-d_2(\bw^*)\Phi^{\HH}W\Phi=\sum_{j=1}^m w_j^*  H_2(x_j);
 \end{align}
 \end{subequations}
 consequently, according to the definition of the set of reference points ${\cal X}_e(\hat \xi)$ with  $|{\cal X}_e(\hat \xi)|\ge 1$, it is true that
 $$0 \preceq H_{\bw^*}=\sum_{j=1}^m w_j^* H(x_j)=\sum_{j:x_j\in {\cal X}_e(\hat \xi)}  w_j^* H(x_j),$$
 where we have used $w_j^*=0$ for any $x_j\not \in {\cal X}_e(\hat \xi)$. Thus  Ruttan's sufficient condition \eqref{eq:Ruttan} is fulfilled. Furthermore, in this case, $\hat \xi$ is the global best approximation of \eqref{eq:bestf}.
 
For  the necessity, suppose that Ruttan's sufficient condition \eqref{eq:Ruttan} is satisfied for an irreducible rational {approximant}  $\hat \xi =\hat  p /\hat  q \in  \scrR_{(n_1,n_2)}$ and $\{z_j\}_{j=1}^t\subseteq {\cal X}_e(\hat\xi)$ and $\{\varpi_j\}_{j=1}^t$. Theorem \ref{thm:RuttanSufOpt} implies that $\hat \xi$ is the global best approximation of \eqref{eq:bestf} with $\eta_\infty=e(\hat \xi)$.
Without loss of generality, assume $z_j=x_j$ for $1\le j\le t$; define a vector $\bw^*\in {\cal S}$ with $ w_j^*=\varpi_j$ if $1\le j\le t$ and $w^*_j=0$ if $m\ge j\ge t+1$. Thus, using the relation \eqref{eq:HwRuttan}, \eqref{eq:RuttanHc} gives
$
A_{\bw^*} -(\eta_\infty)^2 B_{\bw^*}\succeq 0.
$
Moreover, denote by $\hat p(x) = [\psi_0(x),\dots,\psi_{n_1}(x)]\hat \ba$ and $\hat q(x)=[\phi_0(x),\dots,\phi_{n_2}(x)]\hat \bb$; it is true that $0\ne \hat \bc=\left[\begin{array}{c}\hat \ba \\\hat \bb\end{array}\right] \in \bbC^{n_1+n_2+2}$ satisfies 
\begin{align*}
&\hat \bc^{\HH}\left(A_{\bw^*} -(\eta_\infty)^2 B_{\bw^*}\right)\hat \bc\\
=&\sum_{j=1}^tw_j^* \left(|\hat p(x_j)|^2+|f(x_j)\hat q(x_j)|^2-2{\rm Re}(f(x_j)\overline{\hat p(x_j)}\hat q(x_j))-|\eta_\infty\hat q(x_j)|^2\right)\\
=&\sum_{j=1}^tw_j^* |\hat q(x_j)|^2\left(|f(x_j)-\hat \xi(x_j)|^2-\eta_\infty^2\right)=0,
\end{align*} 
where the last equality follows due to $x_j\in  {\cal X}_e(\hat\xi)$. Thus,  $(\eta_\infty)^2$ is the smallest eigenvalue of the pencil $(A_{\bw^*},  B_{\bw^*})$ and $\hat \bc$ is the associated eigenvector. Note $\hat  {\bc^{\HH}}B_{\bw^*}\hat  {\bc}=\sum_{j=1}^t w_j^*|\hat q(x_j)|^2\ne 0$. Normalize $(\hat p,\hat q)$ using $\bw^*\in {\cal S}$ so that $\hat q$ is feasible for the minimization \eqref{eq:rat-d-compt} (i.e., normalize $\hat \bc$ to have $\hat \bc^{\HH}B_{\bw^*}\hat \bc=1$); thus $(\hat p,\hat q)$, by  Proposition \ref{prop:dual_GEP}, achieves the minimum $d_2(\bw^*)$  of \eqref{eq:rat-d-compt} with $\bw=\bw^*$.  As $d_2(\bw^*)$ is also the   smallest eigenvalue of the pencil $(A_{\bw^*},  B_{\bw^*})$, we have \eqref{eq:strongdualityRuttan}. The proof is complete.
\end{proof}

Based on Theorem \ref{thm:strongdualityRuttan}, we have the equivalence between Ruttan's sufficient condition \eqref{eq:Ruttan} and   strong duality \eqref{eq:strongdual}.
\begin{theorem}\label{thm:strongdualityeqvRuttan}
Suppose \eqref{eq:bestf} has a solution. Then Ruttan's sufficient condition \eqref{eq:Ruttan} is satisfied if and only if   strong duality \eqref{eq:strongdual} holds. Thus, the global best  rational approximant $\xi^*\in   \scrR_{(n_1,n_2)}$ of \eqref{eq:bestf} can be computed from the convex dual {problem} \eqref{eq:dualPX} whenever Ruttan's sufficient condition \eqref{eq:Ruttan} is satisfied.
\end{theorem}
\begin{proof}
Let $\bw^*\in{\cal S}$ be the maximizer of the dual {problem} \eqref{eq:dualPX}. 
When Ruttan's sufficient condition  holds, then by \eqref{eq:strongdualityRuttan} of Theorem \ref{thm:strongdualityRuttan}, we have $d_2(\bw^*)=\max_{\bw\in {\cal S}}d_2(\bw)=(\eta_\infty)^2$, i.e.,   strong duality \eqref{eq:strongdual}.

Conversely, let   strong duality \eqref{eq:strongdual} hold and $\xi^*=p^*/q^*\in  \scrR_{(n_1,n_2)}$ be the irreducible solution of \eqref{eq:bestf}. Then $\eta_\infty=e(\xi^*)=\sqrt{d_2(\bw^*)}$. Using the optimality condition in Proposition \ref{prop:dual_GEP},    relations \eqref{eq:Hk} and  \eqref{eq:HwRuttan}, we have 
$$
0\preceq H_{\bw^*}=A_{\bw^*}-(e(\xi^*))^2B_{\bw^*}=\sum_{j=1}^m w_j^* H(x_j)=\sum_{j:x_j\in {\cal X}_e(\xi^*)}  w_j^* H(x_j),
$$
where, based on  the complementary condition in Theorem \ref{thm:complement}, we have used $w_j^*=0$ for any $x_j\not \in {\cal X}_e(\xi^*)$. This completes the proof.
\end{proof}

As the dual {problem} \eqref{eq:dualPX} is a convex optimization {problem}, one can expect the global solution can be computed effectively.  This means that, in theory, solving \eqref{eq:bestf} through its dual problem \eqref{eq:dualPX} can ensure  to obtain the global best rational approximation of \eqref{eq:bestf} under Ruttan's sufficient condition \eqref{eq:Ruttan}. Note that \eqref{eq:strongdualityRuttan} provides a way to check Ruttan's sufficient condition \eqref{eq:Ruttan} via solving the dual problem globally.  On the other hand, even it is claimed in \cite{isth:1993} that  ``{\it his condition turns out to be effective in most examples encountered in our numerical experiments}'' and can be a necessary condition for some special case (see e.g., \cite[Theorem 5]{this:1993}), it has been shown in \cite{isth:1993,this:1993} that Ruttan's sufficient condition \eqref{eq:Ruttan} generally is not necessary for a solution of \eqref{eq:bestf}. Indeed, this is expected as we use a convex optimization {problem}  to approximate the original non-convex problem \eqref{eq:bestf}; whenever Ruttan's sufficient condition \eqref{eq:Ruttan} is not fulfilled at the solution of \eqref{eq:bestf}, the computed  rational {approximant} from the dual \eqref{eq:dualPX} could at most be a local best approximation of \eqref{eq:bestf}. Nevertheless, our numerical experiments in section \ref{sec_Numerical} indicate that Ruttan's sufficient condition \eqref{eq:Ruttan} holds frequently in our numerical test problems.

For the uniqueness of the global best rational solution to \eqref{eq:bestf}, under Ruttan's sufficient condition \eqref{eq:Ruttan},  \cite[Theorem 4]{this:1993} shows that when the Hermitian  matrix $H$ in \eqref{eq:RuttanHc}, or equivalently,  $H_{\bw^*}$, has rank $\rank(H_{\bw^*})=n_1+n_2+1$, then the  global best solution is unique. This corresponds to the case when $d_2(\bw^*)=(e(\xi^*))^2$ is a simple eigenvalue of the matrix pencil $(A_{\bw^*}, B_{\bw^*})$. For this situation, in the next section, we shall provide further information, the gradient and the Hessian, of the dual function $d_2(\bw)$ with respect to $\bw$.  

\section{The gradient and Hessian of $d_2(\bw)$}\label{sec_gradHess} 
\begin{proposition}\label{prop:grad}
For $\bw>0$,  let  $d_2(\bw)$ be the smallest eigenvalue of the Hermitian positive semi-definite generalized eigenvalue problem \eqref{eq:dual_GEP}, and $\bc(\bw)=\left[\begin{array}{c} \ba(\bw) \\ \bb(\bw)\end{array}\right] \in \bbC^{n_1+n_2+2}$ be the associated eigenvector. Denote $\bp =[p_1,\dots,p_m]^{\T}=\Psi \ba(\bw)\in \bbC^m$ and $\bq  =[q_1,\dots,q_m]^{\T}=\Phi \bb(\bw)\in \bbC^m$. 
If $d_2(\bw)$ is a simple eigenvalue, then 
\begin{itemize}
\item[(i)]
$d_2(\bw)$ is differentiable at $\bw$ and its gradient is 
\begin{align} 
\nabla d_2(\bw)&=\left[\begin{array}{c}|f_1 q_1 - p_1 |^2-d_2(\bw)| q_1 |^2 \\|f_2 q_2 - p_2 |^2-d_2(\bw)| q_2 |^2  \\\vdots \\|f_m q_m - p_m |^2-d_2(\bw)| q_m |^2 \end{array}\right] =:|F \bq -\bp |^2-d_2(\bw)|\bq |^2\in \bbR^m; \label{eq:gradd2}
\end{align}
\item[(ii)] the symmetric Hessian $\nabla^2 d_2(\bw)\in \bbR^{m\times m} $ is 
\begin{equation}\label{eq:Hess}
\nabla^2 d_2(\bw)=-2{\rm Re}\left(R_3(A_{\bw}-d_2(\bw)B_{\bw})^{\dag}R_3^{\HH}\right)-\nabla d_2(\bw)(|\bq |^2)^{\T}-|\bq |^2 (\nabla d_2(\bw))^{\T},
\end{equation}
where $R_1=\diag(F\bq -\bp )\in \bbC^{m\times m}$, $R_2=\diag(\bq )\in \bbC^{m\times m}$, $$R_3^{\HH}=E_1^{\HH}R_1-d_2(\bw)E_2^{\HH}R_2-B_{\bw} \bc(\bw) (\nabla d_2(\bw))^{\T}\in \bbC^{(n_1+n_2+2)\times m},$$ $E_1=[-\Psi, F\Phi]\in \bbC^{m\times (n_1+n_2+2)}$ and $E_2=[0, \Phi]\in \bbC^{m\times (n_1+n_2+2)}$.
\end{itemize}
\end{proposition}
  
{The proof of Proposition \ref{prop:grad} is given in Appendix \ref{AppendixA}.} Even  though \eqref{eq:Hess} gives the explicit form of   $\nabla^2 d_2(\bw)$, the computation involves the ill-conditioned Vandermonde basis matrices $\Psi$ and $\Phi$. Similar to our treatment of computing $d_2(\bw)$ in \eqref{eq:pqvector} and the evaluation stage \eqref{eq:new_pqvector}, we next provide an alternative formulation of $\nabla^2 d_2(\bw)$; the new way uses the orthonormal basis matrices $Q_p$ and $Q_q$ of ${\rm span}(W^{\frac12}\Psi)$ and ${\rm span}(W^{\frac12}\Phi)$, respectively, and the computed vector $\what \bb$  in \eqref{eq:pqvector}. {The proof of Corollary \ref{cor:Hess2} is given in Appendix \ref{AppendixA}.}
\begin{corollary}\label{cor:Hess2}
Under notations and conditions of Proposition \ref{prop:grad}, let  $Q_p\in \bbC^{m\times (n_1+1)}$ and $Q_q\in \bbC^{m\times (n_2+1)}$ be the orthonormal basis matrices  of ${\rm span}(W^{\frac12}\Psi)$ and ${\rm span}(W^{\frac12}\Phi)$, respectively. Let $\sqrt{d_2(\bw)}$ be also the smallest singular value of $(Q_p^{\perp})^{\HH}FQ_q\in \bbC^{(m-n_1-1)\times (n_2+1)}$ (or  equivalently  $(I-Q_pQ_p^{\HH})FQ_q\in \bbC^{m\times ( n_2+1)}$) and $\what \bb$ be the associated right singular vector. Then the symmetric Hessian $\nabla^2 d_2(\bw)\in \bbR^{m\times m}$ can be computed as
\begin{equation}\nonumber
\nabla^2 d_2(\bw)=-2 {\rm Re}\left(\what R_3 D^{\dag}\what R_3^{\HH}\right) -\nabla d_2(\bw)(|\bq |^2)^{\T}-|\bq |^2 (\nabla d_2(\bw))^{\T},
\end{equation}
where $D=\left[\begin{array}{cc}I_{n_1+1} & -S_{qp}^{\HH} \\-S_{qp} & S_F-d_2(\bw) I_{n_2+1}\end{array}\right]$,
$R_1=\diag(F\bq -\bp )\in \bbC^{m\times m}$, $R_2=\diag(\bq )\in \bbC^{m\times m}$, and
$
\what R_3= W^{-\frac12}[-R_1^{\HH}Q_p,(R_1^{\HH}F-d_2(\bw)R_2^{\HH})Q_q-W^{\frac12}\nabla d_2(\bw)\what\bb^{\HH}]\in \bbC^{m\times (n_1+n_2+2)}.
$
\end{corollary}
 

\section{Lawson's iteration}\label{sec_lawson} 
\subsection{Lawson's  iteration for  the dual {problem}}
As has been explored in \cite{yazz:2023} (see also \cite{clin:1972,rice:1969}), the traditional Lawson's iteration  \cite{laws:1961}   for the linear minimax approximation can be viewed as a monotonically ascent convergent method for the associated Lagrange dual {problem} with   strong duality. According to our dual {problem} \eqref{eq:dualPX}, a natural extension of Lawson's iteration  to the rational case is to use the residual $\Bf-\xi^{(k)}(\bx)$ at the current weight vector $\bw^{(k)}$, i.e., the dual variables, and update it. The same idea has been applied in the literature (e.g., \cite{bama:1970,coop:2007,loeb:1957,hoka:2020,sako:1963}), yielding several versions of Lawson's iteration. We will discuss the differences between these versions in section \ref{subsec_others} and provide  numerical experiments in section \ref{sec_Numerical}.
Our proposed Lawson's iteration is summarized in Algorithm \ref{alg:Lawson};  {specifically, we  refer to Algorithm \ref{alg:Lawson}   as {\tt d-Lawson} to reflect that it solves the dual problem of the rational minimax approximation.}

\begin{algorithm}[h!!!]
\caption{A rational Lawson's iteration ({\tt d-Lawson}) for \eqref{eq:bestf}} \label{alg:Lawson}
\begin{algorithmic}[1]
\renewcommand{\algorithmicrequire}{\textbf{Input:}}
\renewcommand{\algorithmicensure}{\textbf{Output:}}
\REQUIRE Given samples $\{(x_j,f_j)\}_{j=1}^m$ and $0\le n_1+n_2+2\le m$ with $x_j\in \Omega$, a relative tolerance for strong duality $\epsilon_r>0$, the maximum number $k_{\rm maxit}$ of iterations, a new vector $\by=[y_1,\dots,y_{\wtd m}]^{\T}\in \bbC^{\wtd m}$ of nodes; 
\ENSURE   the evaluation vector $\xi^*(\by)\in \bbC^{\wtd m}$ of the (approximation)  $\xi^*\in  \scrR_{(n_1,n_2)}$ of \eqref{eq:bestf} at the new nodes vector $\by$;
        \smallskip

\STATE  (Initialization) Let $k=0$; choose $0<\bw^{(0)}\in {\cal S}$  and a  tolerance $\epsilon_w$ for the weights;
\STATE (Filtering) Remove nodes $x_i$ with $w_i^{(k)}<\epsilon_w$;
\STATE Compute  $d_2(\bw^{(k)})$ and the associated vector $\xi^{(k)}(\bx)=\bp./\bq $  according to \eqref{eq:pqvector};

\STATE (Stop rule and evaluation)  Stop either if $k\ge k_{\rm maxit}$ or 
\begin{equation}\label{eq:stop}
\epsilon(\bw^{(k)}):=\left|\frac{\sqrt{d_2(\bw^{(k)})}-e(\xi^{{(k)}})}{e(\xi^{{(k)}})}\right|<\epsilon_r,~~{\rm where}~~e(\xi^{{(k)}})=\|\Bf-\xi^{(k)}(\bx)\|_\infty,
\end{equation}
and compute $\xi^*(\by)=p^*(\by)./q^*(\by)$ according to \eqref{eq:new_pqvector};

\STATE (Updating weights) Update the weight vector $\bw^{(k+1)}$ according to   ($0<\beta\le 1$)
\begin{equation}\nonumber
w_j^{(k+1)}=\frac{w_j^{(k)}\left|f_j-\xi_j^{(k)}\right|^{\beta}}{\sum_{i}w_i^{(k)}\left|f_i-\xi_i^{(k)}\right|^{\beta}},~~\forall j,
\end{equation} 
and   goto Step 2 with $k=k+1$.
\end{algorithmic}
\end{algorithm}  

We make the following remarks for {{\tt d-Lawson} (Algorithm \ref{alg:Lawson})}.  
\begin{remark}\label{rmk:Alg1}
\item [(1).] The vector $\bw^{(0)}=\be/m$ can be a good choice of the initial weight vector.

\item[(2).] The filtering procedure in Step 2 is based on Theorem \ref{thm:complement}. Indeed, whenever    strong duality (i.e., Ruttan's sufficient condition \eqref{eq:Ruttan}) holds and $\min_{j\in {\cal X}_e(\xi^*)} w_j^*>0$, then  with $0\le \epsilon_w<\min_{j\in {\cal X}_e(\xi^*)} w_j^*$,  the algorithm generally is able to compute the   rational minimax approximation $\xi^*=p^*/q^*$. Our numerical experiments suggest a tiny $\epsilon_w$ (for example $10^{-30}\sim10^{-50}$).

\item[(3).] The stop rule \eqref{eq:stop} is based on strong duality. When \eqref{eq:stop} is met, we can guarantee that the obtained solution is an approximation of the best solution of \eqref{eq:bestf}, and also know that Ruttan's sufficient condition \eqref{eq:Ruttan} holds. Our numerical experiments in section \ref{sec_Numerical} indicate that \eqref{eq:stop} can be numerically satisfied (for a given level of $\epsilon_r$) in many test problems.

\item[(4).] Another interesting observation from our numerical experiments is that the sequence $\left\{\sqrt{d_2(\bw^{(k)})}\right\}_{k=0}^{\infty}$ of dual objective {function values} usually    increases monotonically whenever the rounding error does not dominate; moreover, for any $k$, we have the weak duality $\sqrt{d_2(\bw^{(k)})}\le e(\xi^{(k)})$ and the gap  (and hence $\epsilon(\bw^{(k)})$)  gets smaller as $k$ increases. 

\item[(5).] As reported in \cite{fint:2018}, a small  {\it Lawson exponent} $\beta$ (for example $\beta=1/4$) can make Lawson's iteration much more robust, while a  large one  (for example $\beta=1$) can {lead to faster}  convergence. Even though our numerical experiments indicate the monotonic convergence with respect to  $d_2(\bw)$,  we currently do not have a theoretical guarantee of the convergence for a constant $\beta>0$. Nevertheless, as the relative duality gap $\epsilon(\bw^{(k)})$ is a monitor for the accuracy of the approximation, we can record and return the approximant associated with the smallest  $\epsilon(\bw^{(k)})$ among the previous iterations. This practical and robust strategy was also used in the stabilized SK iteration\footnote{For the code of S-SK, see   \url{https://github.com/jeffrey-hokanson/polyrat}.}.

\item[(6).] The output of   {{\tt d-Lawson}   (Algorithm \ref{alg:Lawson})} is the evaluation vector $\xi^*(\by)\in \bbC^{\wtd m}$ of the (approximation)  $\xi^*\in  \scrR_{(n_1,n_2)}$ of \eqref{eq:bestf} at the new nodes vector $\by$. Whenever the rational {approximant} representation $\xi^*=p^*/q^*$ is desired, by using the evaluation vector $\xi^*(\by)$, certain stable polefinding approaches such as  the function  \textsf{ratfun}  in \cite{itna:2018}, can be employed  to compute the roots of $p^*$ and $q^*$, or their  corresponding coefficient vectors in the monomial basis. In {\tt Chebfun}\footnote{\url{http://www.Chebfun.org}} \cite{drht:2014}, a similar function is \textsf{ratinterp}.
\end{remark}

\subsection{Other iteratively reweighted least-squares (IRLS) iterations}\label{subsec_others}
In the literature, one of the earliest IRLS methods that introduces a sequence of weighted least-squares problems to approach a rational approximation is the Loeb algorithm  \cite{loeb:1957}. The same method was also proposed in \cite{sako:1963}, known as the SK iteration by Sanathanan and  Koerner. Specifically, in the Loeb algorithm, the reciprocal of the current denominator $\Phi\bb^{(k)}$ is used as a weight vector to compute  the coefficient vectors $\ba^{(k+1)}$ and $\bb^{(k+1)}$:
\begin{equation}\label{eq:SK}
\left[\begin{array}{c}\ba^{(k+1)} \\\bb^{(k+1)} \end{array}\right]=\argmin_{\|\ba\|_2^2+\|\bb\|_2^2=1,~\bb\ne \mathbf{0}}\left\|\left(\diag(\Phi\bb^{(k)})\right)^{-1}[-\Psi, F\Phi] \left[\begin{array}{c}\ba \\\bb\end{array}\right]\right\|_2.
\end{equation}
Comparing with our dual function $d_2(\bw)$ given in \eqref{eq:rat-d-compt}, we know that at the $k$th iteration, the associated $\bw$ is the normalized vector of  $\frac{1}{|\Phi\bb^{(k)}|^2}$ (i.e., $\sqrt{W}=\left(\diag(\Phi\bb^{(k)})\right)^{-1}$ in \eqref{eq:rat-d-compt}). Note that the scaled weight vector $\frac{1}{m|\Phi\bb^{(k)}|^2}$ also leads to the same solution in \eqref{eq:SK}; correspondingly, the weight vector $\bw=\frac{1}{m|\Phi\bb^{(k)}|^2}$ fulfills the constraint: 
$
\|\sqrt{W}\Psi\bb^{(k)}\|_2=1
$ 
of \eqref{eq:rat-d-compt}. In this sense, it seems that the Loeb algorithm (i.e, the SK algorithm) chooses a particular feasible weight vector for the constraint of \eqref{eq:rat-d-compt} and   treats its objective function and the constraint alternatively. Unfortunately, as claimed in \cite[Section 2.2]{coop:2007}   `{\it there is no convergence proof for the algorithm, and when used for least-squares problems even if the algorithm does converge, it is almost certain not to converge to a best approximation}.'  See also the numerical {results} reported in  \cite{bama:1970}.  

In a recent work \cite{hoka:2020}, a stabilized SK iteration  was proposed. Specifically, the weights are updated by
$$
w_j^{(k+1)}=\frac{w_j^{(k)}}{\left|[\Phi\bb^{(k)}]_j\right|},~~w_j^{(k+1)}=\frac{w_j^{(k+1)}}{\sum_{i}w_i^{(k+1)}},\quad \forall j,
$$
and solve 
\begin{equation}\label{eq:S-SK}
\left[\begin{array}{c}\ba^{(k+1)} \\\bb^{(k+1)} \end{array}\right]=\argmin_{\|\ba\|_2^2+\|\bb\|_2^2=1,~\bb\ne \mathbf{0}}\left\|  \diag(\bw^{(k+1)})[-\Psi, F\Phi] \left[\begin{array}{c}\ba \\\bb\end{array}\right]\right\|_2.
\end{equation}
In computing \eqref{eq:S-SK}, the stabilized SK iteration uses the weighted Arnoldi process  to form orthogonal basis matrices for $\diag(\bw^{(k+1)})\Psi$ and $\diag(\bw^{(k+1)}) F\Phi$, and also calls the Arnoldi process for evaluation at new points (see \cite{brnt:2021} and also \eqref{eq:pqvector} and \eqref{eq:new_pqvector}). It is reported in  \cite{hoka:2020} that the stabilized SK iteration performs more stably, and generally returns a smaller least-squares residual norm than the standard SK iteration. However, it can cycle and does not monotonically decrease the least-squares residual norm; furthermore, the computed solution is generally not the   rational minimax  approximant of \eqref{eq:bestf}. We shall report numerical results in section \ref{sec_Numerical}.

A remarkable recent method for {computing} the rational approximation is the AAA (stands for {\it ``adaptive Antoulas-Anderson''}) algorithm proposed by  Nakatsukasa,  S\`ete and  Trefethen in \cite{nase:2018}. To handle the ill-conditioned Vandermonde basis, the method represents the rational approximation in barycentric form, where the associated support points are selected iteratively in a greedy way to  avoid instabilities, and then computes the corresponding   coefficients in a least-squares sense over the nodes excluding the supporting points. Even though the AAA algorithm is generally in a least-squares mode, it is also a very efficient approach to compute a good approximation of the minimax approximation \eqref{eq:bestf}. Moreover, Lawson's idea can be further incorporated and  AAA-Lawson \cite{fint:2018,natr:2020} turns out to be an effective method for  \eqref{eq:bestf}.  In the next section, we shall carry out numerical tests and report on numerical results from AAA,  AAA-Lawson, the stabilized SK iteration,  {the rational Krylov fitting (RKFIT)  \cite{begu:2017,gogu:2021}, and {\tt d-Lawson} (Algorithm \ref{alg:Lawson})}. Our numerical experiments demonstrate that {{\tt d-Lawson}} is an effective competitive   method for computing the   minimax approximation  of \eqref{eq:bestf}.

\section{Numerical results}\label{sec_Numerical} 
\def\zhu#1{\textcolor{red}{#1} }

We implement {\tt d-Lawson}\footnote{The MATLAB code of {\tt d-Lawson} is available at \url{https://ww2.mathworks.cn/matlabcentral/fileexchange/167176-d-lawson-method}.}  in MATLAB R2018a and conduct numerical experiments on a 13-inch {MacBook Air} with {an} M2 chip and 8Gb memory. The unit machine roundoff {is} $\tu=2^{-52}\approx 2.2\times 10^{-16}$. Initially, we set $\bw^{(0)}=\be/m$; particular parameters  of {\tt d-Lawson} (Algorithm \ref{alg:Lawson}) are set by default as $k_{\rm maxit}=40$,   $\epsilon_r=10^{-5}$ and  $\beta=1$, and the algorithm returns the computed solution given in (5) Remark   \ref{rmk:Alg1} if it is not otherwise specified.   We denote by {{\tt d-Lawson}($k_{\rm maxit}$)}, S-SK($k_{\rm maxit}$), AAA($k_{\rm maxit}$) and  {RKFIT($k_{\rm maxit}$)},  the versions of {\tt d-Lawson}, the stabilized SK\footnote{For the stop rule of S-SK, we choose $\|\sqrt{W^{(k)}}(\xi^{(k)}(\bx)-\xi^{(k-1)}(\bx))\|_2\le  10^{-11} $  which is also used  in \url{https://github.com/jeffrey-hokanson/polyrat}.} (S-SK), AAA-Lawson\footnote{For the AAA algorithm \cite{nase:2018,fint:2018} in the package {\tt Chebfun},  one can set the maximum number of Lawson's iteration $k$ by calling {\tt aaa(f,x, 'lawson',k)}; we shall denote this version by AAA($k$). In particular, AAA(0) represents the basic AAA algorithm \cite{nase:2018} without Lawson's iteration.} (call {\tt aaa} in {\tt Chebfun}) and RKFIT\footnote{The code of RKFIT is available at \url{http://rktoolbox.org/}. In our testing, the initial poles are randomly drawn from the standard normal distribution.}  with the specific maximum number of iterations  $k_{\rm maxit}$,  respectively. 
 
\subsection{Numerical experiments on real cases}\label{subsec-real}
{We first evaluate} the proposed method, {\tt d-Lawson}, and compare the numerical performance with AAA, AAA-Lawson, S-SK iteration {and RKFIT for real-valued functions on subsets of the real line}.  For this purpose, we choose two basic real functions without singularity and other two with singularities in given intervals.

\subsubsection{Basic function}
We   evaluate the methods on the following two basic functions 
$$
f_1(x)= |x|,   ~x \in[-1,1],\quad f_2(x)= \sqrt{x},    ~x \in \left[10^{-8}, 1\right].
$$
The absolute function is a standard test example that contains a nonsmooth point  $x=0$;    the   rational minimax approximation of $\sqrt{x}$ is a classical Zolotarev problem \cite[Chapter 9]{akhi:1990} and is used for testing  in  \cite{fint:2018}.  For each function, we compute the   minimax {approximant} 
$\xi(x) = p(x)/q(x)\in \scrR_{(n_1,n_2)}$ on a set of samples  $\{(x_j,f(x_j))\}_{j=0}^m,$
where $\{x_j\}_{j=0}^m$ are equally spaced in their respective intervals with  
 $m = 2000.$

As a representative example, we present detailed numerical results from $f_1(x)=|x|$ to demonstrate
\begin{itemize}
\item[ (i)] the curve of the error function (i.e., the equioscillation property \cite[Theorem 24.1]{tref:2019a}), 
\item[(ii)] the monotonic convergence of Algorithm \ref{alg:Lawson} with respect to the dual objective function value $d_2(\bw)$, 
\item [(iii)] the weak duality $\sqrt{d_2(\bw^{(k)})}\le e(\xi^{(k)})$ as well as strong duality \eqref{eq:strongdualityRuttan} as $k$ increases,  
\item[(iv)] the performance of the filtering procedure in Step 2 of Algorithm \ref{alg:Lawson} and verify the theoretical result of Theorem \ref{thm:complement}. 
\end{itemize} 
To demonstrate (i)$\sim$(iii), we set  $\epsilon_w=0$ (i.e., the filtering procedure in Step 2 of Algorithm \ref{alg:Lawson} is inactive) for {\tt d-Lawson}.    Figure \ref{fig:7-exam1-1} gives  the minimax error curves  ($e(\xi)=\|\Bf-\xi(\bx)\|_\infty$ defined in \eqref{eq:exi} is the   maximum  error of the approximant $\xi$) for approximating $f_1(x)$ with  $\xi(x) = p(x)/q(x)\in \scrR_{(4,4)}$. It can be seen that the equioscillation property is clearly observed in  {\tt d-Lawson}(20) and  AAA(20), but not in S-SK(20) and {in RKFIT(20)}.

\begin{figure}[h!!!] 
 \begin{center}
\includegraphics[width=0.99\linewidth,height=0.2\textheight]{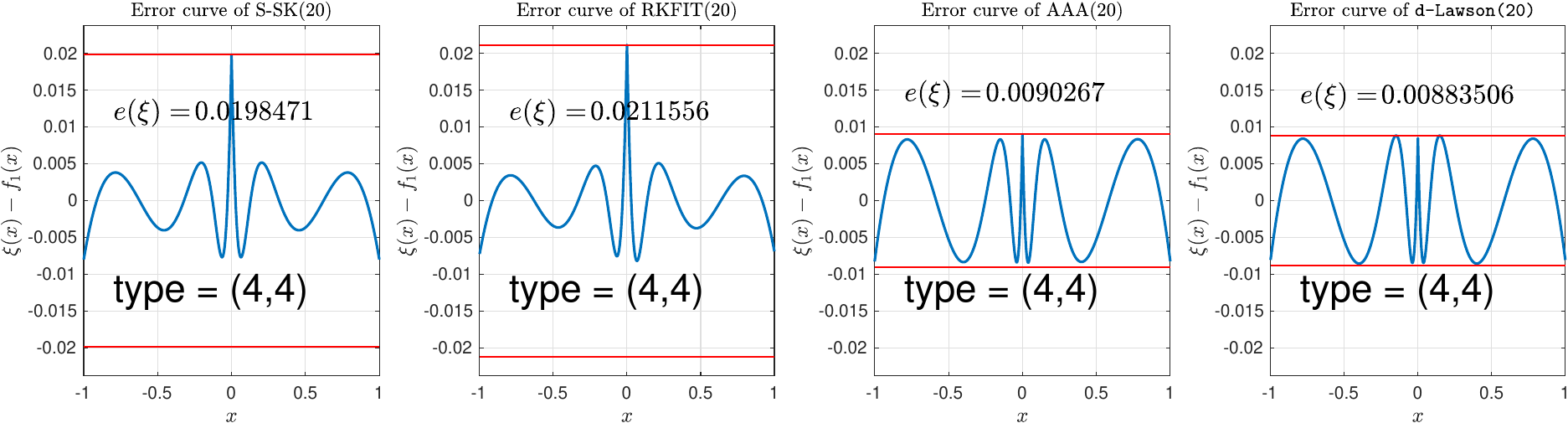} 
\caption{{\small Error curves  (blue) for type (4,4) approximant $\xi(x)  \approx f_1(x)$ by the {four} methods.} }
\label{fig:7-exam1-1}
 \end{center}
\end{figure}

To demonstrate (ii) and (iii),  we plot the sequences of $\left\{e(\xi^{(k)})\right\}$ and $\left\{\sqrt{d_2(\bw^{(k)})}\right\}$ in Figure \ref{fig:7-exam1-2} for {\tt d-Lawson} (Algorithm \ref{alg:Lawson}) after  a fixed number $k$ of iterations (here the computed approximant is not the one associated with the smallest $\epsilon(\bw^{(k)})$ described in  (5) Remark \ref{rmk:Alg1}). It   demonstrates the monotonic convergence  of  $\left\{\sqrt{d_2(\bw^{(k)})}\right\}$, the weak duality and strong duality (i.e., Ruttan's sufficient condition \eqref{eq:Ruttan}). The latter implies that a best   rational approximant in $\scrR_{(4,4)}$ satisfying Ruttan's sufficient condition \eqref{eq:Ruttan} exists.

\begin{figure}[h!!!]
 \begin{center}
\includegraphics[width=0.85\linewidth,height=0.22\textheight]{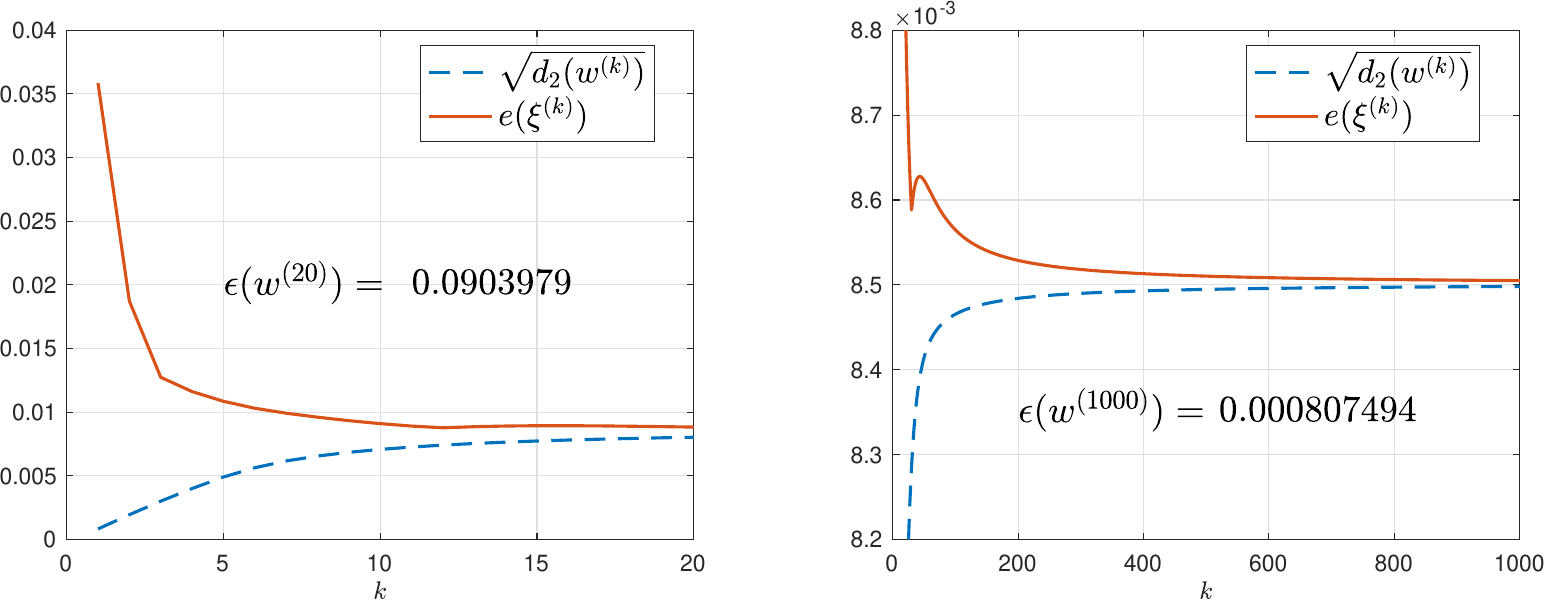} 
\caption{ \small The sequences of $\left\{e(\xi^{(k)})\right\}$  and $\left\{\sqrt{d_2(\bw^{(k)})}\right\}$ for {{\tt d-Lawson}  (Algorithm \ref{alg:Lawson})} for approximating $|x|$ with type ${(4,4)}$.}
\label{fig:7-exam1-2}
\end{center}
\end{figure}

We further demonstrate (iv), i.e., the filtering procedure in Step 2 of {\tt d-Lawson} (Algorithm \ref{alg:Lawson}) and  the theoretical result of Theorem \ref{thm:complement}. To this end, we choose three tolerances   $\epsilon_w=0,~\epsilon_w=10^{-40}$, $\epsilon_w=10^{-30}$, and run {\tt d-Lawson}.   In Figure \ref{fig:7-exam1-3}, we 
provide the error curves, the maximum  error, the relative duality gap $\epsilon(\bw^{(k)})$   associated with the weight vector $\bw^{(k)}$ and the number of {final remaining nodes}. It can be seen that the three approximants share the same error curve with the maximum error $e(\xi)\approx 8.51\times 10^{-3}$ and the relative duality gap $\epsilon(\bw^{(k)})\approx 8.07\times 10^{-4}$. 
 
\begin{figure}[h!!!]
\begin{center}
\includegraphics[width=0.95\linewidth,height=0.2\textheight]{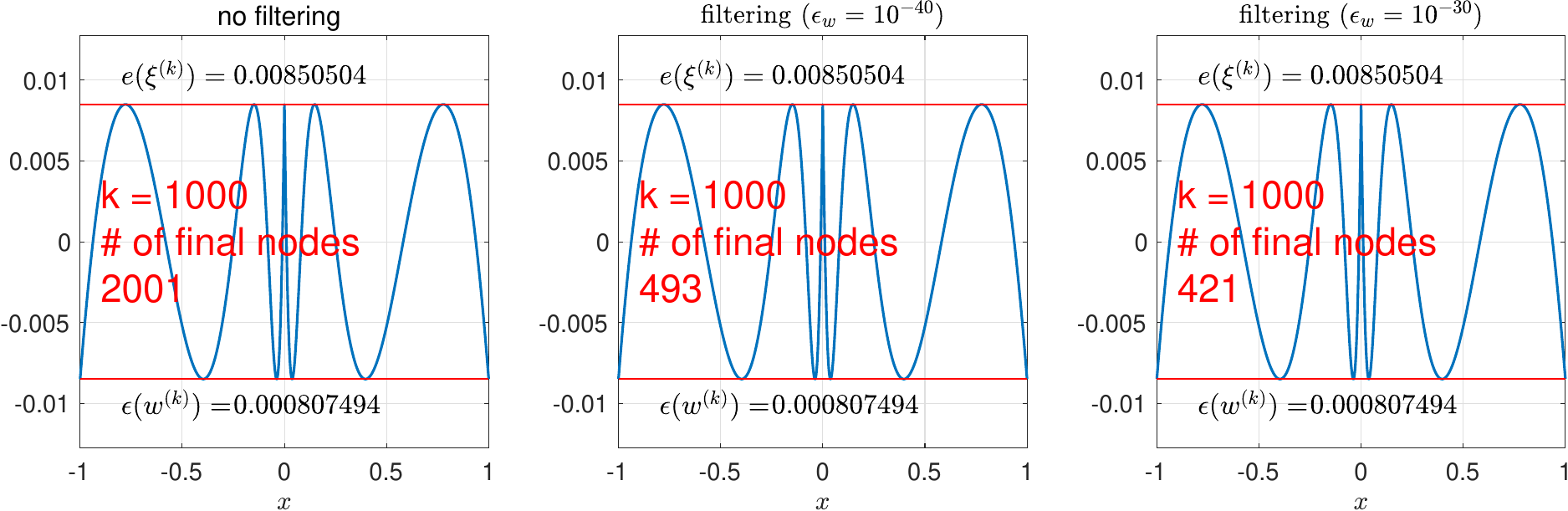} 
\caption{ \small    Error curves  of type ${(4,4)}$ to approximate $|x|$ with $\epsilon_w=0, 10^{-40},~10^{-30}$, respectively.
}
\label{fig:7-exam1-3}
\end{center}
\end{figure}
 
Setting $\epsilon_w=0$ for {{\tt d-Lawson}  (Algorithm \ref{alg:Lawson})}, we  extend the numerical test on $f_1(x)=|x|$ and $f_2(x)=\sqrt{x}$ by  varying  $n_1$ and $n_2$. {In order to clearly demonstrate the performance in Figure \ref{fig:7-exam1-4}, we only  plot the maximum errors  $e(\xi)=\|\Bf-\xi(\bx)\|_\infty$ of AAA, AAA-Lawson, S-SK and {\tt d-Lawson} versus various diagonal types $(n,n)$, and  report results from RKFIT in Tables \ref{tab:7-exam1} and \ref{tab:7-exam2}.}  It is observed that {\tt d-Lawson}(40)  performs stably in computing the minimax approximants. AAA(40) is also effective and very efficient, but in some special cases (e.g., the types of $(12,12)$ and  $(18,18)$),  the computed solutions are only {near best}. S-SK and {RKFIT} can return good rational approximants, which generally are not the minimax solutions as  the corresponding  error curves do not {satisfy} the equioscillation property.

\begin{figure}[h!!!]
\begin{minipage}[t]{1\textwidth} 
\begin{center}\hskip-5 mm
{ \subfigure{
		\includegraphics[width=0.49\columnwidth,height=0.3\textheight]{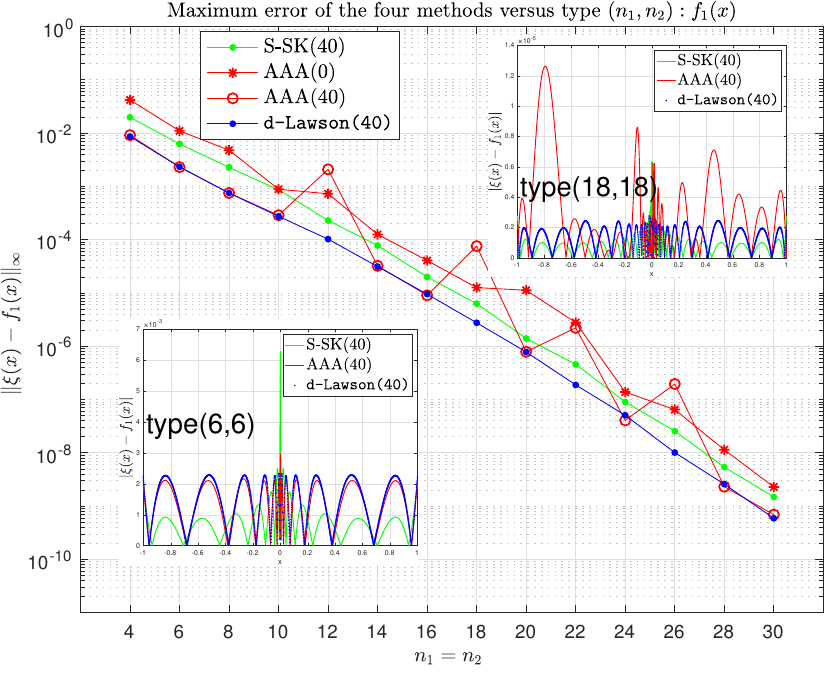}  
	}
	\hspace{0cm}\subfigure{
		\includegraphics[width=0.49\columnwidth,height=0.3\textheight]{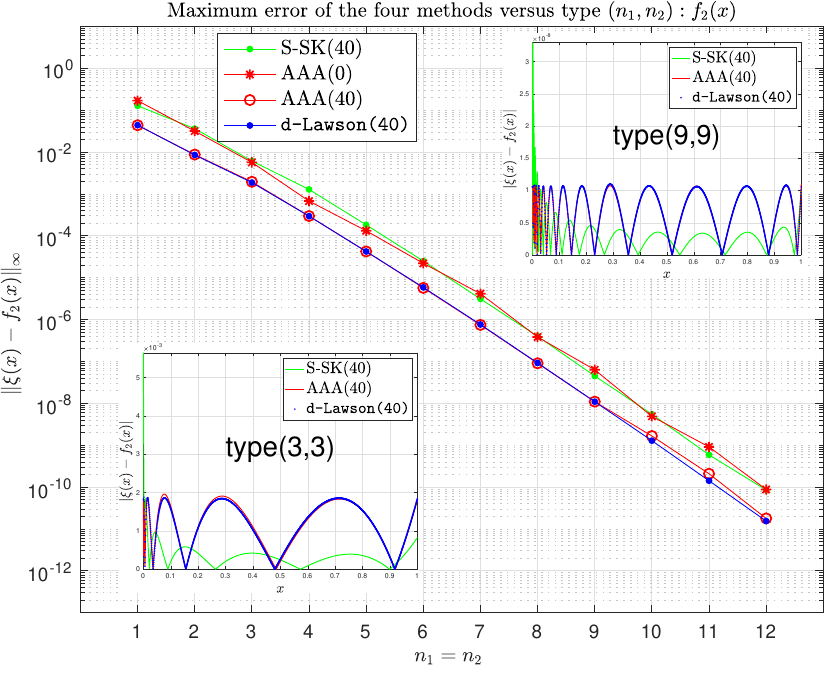}  
	}
  }
\caption{ \small  The maximum error $e(\xi)=\|\Bf-\xi(\bx)\|_\infty$ of
 the four methods with respect to various types $(n,n)$ in approximating $f_1(x)=|x|$ (left) and $f_2(x)=\sqrt{x}$ (right). Error curves with particular types $(n,n)$  are also plotted.
}\label{fig:7-exam1-4}
\end{center}
\end{minipage}
\small\captionof{table}{\small The relative duality gap $\epsilon(\bw^{(k)})$ and maximum errors for approximating $|x|$.}
\begin{center}\tabcolsep0.08in
\begin{tabular}[h!!!]{|c|c|c|c|c|c|c|}
\hline  
 \tabcolsep0.1in
  & $\epsilon(\bw^{(k)})$ & $\sqrt{d_{2}(\bw^{(k)})}$ &  $ e_{\tt d-Lawson(40)}(\xi) $   &      
  $e_{\rm S-SK(40)}(\xi) $  &   $e_{\rm AAA(40)}(\xi) $   &  ${e_{\rm RKFIT(40)}(\xi) }$ \\  \hline
(4,4) &        0.029107 &      8.3752e-03 &      8.6262e-03 &      1.9847e-02 &      9.1480e-03 &      2.1156e-02 \\ \hline 
(8,8) &        0.054627 &      7.1036e-04 &      7.5141e-04 &      2.2843e-03 &      7.5278e-04 &      2.7722e-03 \\ \hline 
(12,12) &        0.157363 &      8.6457e-05 &      1.0260e-04 &      2.2970e-04 &      2.0890e-03 &      3.9191e-04 \\ \hline 
(16,16) &        0.216452 &      7.4682e-06 &      9.5313e-06 &      2.0034e-05 &      9.0516e-06 &      1.5327e-03 \\ \hline 
(20,20) &        0.401854 &      4.6816e-07 &      7.8268e-07 &      1.4031e-06 &      7.8967e-07 &      8.4410e-04 \\ \hline 
(24,24) &        0.409439 &      3.0056e-08 &      5.0895e-08 &      9.0193e-08 &      4.0732e-08 &      2.3178e-07 \\ \hline 
(28,28) &        0.340707 &      1.7076e-09 &      2.5901e-09 &      5.4073e-09 &      2.3281e-09 &      9.9814e-09 \\ \hline 
\end{tabular}
\label{tab:7-exam1}
\end{center}

\begin{center} \tabcolsep0.08in
\captionof{table}{\small The relative duality gap $\epsilon(\bw^{(k)})$ and maximum errors for approximating $\sqrt{x}$.} 
\begin{tabular}[h!!!]{|c|c|c|c|c|c|c|}
\hline  
 \tabcolsep0.1in
  & $\epsilon(\bw^{(k)})$ & $\sqrt{d_{2}(\bw^{(k)})}$ &  $ e_{\tt d-Lawson(40)}(\xi) $   &      
  $e_{\rm S-SK(40)}(\xi) $  &   $e_{\rm AAA(40)}(\xi) $   &  ${e_{\rm RKFIT(40)}(\xi) }$ \\  \hline
(1,1) &        0.020489 &      4.3208e-02 &      4.4111e-02 &      1.2792e-01 &      4.3897e-02 &      1.3141e-01 \\ \hline 
(3,3) &        0.031519 &      1.8219e-03 &      1.8812e-03 &      6.1498e-03 &      1.9605e-03 &      1.0191e-02 \\ \hline 
(5,5) &        0.026301 &      4.1306e-05 &      4.2422e-05 &      1.8375e-04 &      4.2540e-05 &      2.0911e-04 \\ \hline 
(7,7) &        0.066449 &      7.1439e-07 &      7.6524e-07 &      3.1287e-06 &      7.4933e-07 &      3.5369e-06 \\ \hline 
(9,9) &        0.039936 &      1.0545e-08 &      1.0984e-08 &      4.5100e-08 &      1.1017e-08 &      5.1302e-08 \\ \hline 
(11,11) &        0.050260 &      1.3539e-10 &      1.4256e-10 &      5.9256e-10 &      2.0898e-10 &      6.8057e-10 \\ \hline 
\end{tabular}
\label{tab:7-exam2} 
\end{center}
\end{figure}

To demonstrate strong duality with different $(n_1,n_2)$ for $f_1(x)=|x|$ and $f_2(x)=\sqrt{x}$,  in Tables \ref{tab:7-exam1} and \ref{tab:7-exam2}, respectively, we present the final  relative duality gaps  $\epsilon(\bw^{(k)})$ for {{\tt d-Lawson}  (Algorithm \ref{alg:Lawson})}. 
Reported are also the maximum errors from other methods. 
{These} numerical experiments  indicate that Ruttan's sufficient condition \eqref{eq:Ruttan} holds in all these cases (by increasing $k$, we observed that $\epsilon(\bw^{(k)})$ from {\tt d-Lawson}  decreases).

{A further verification of {\tt d-Lawson} is to compare it with the {\tt minimax} algorithm \cite{fint:2018} in the  package {\tt Chebfun}.  {\tt minimax} is a method based on the Remez algorithm  using the barycentric representation  for the rational approximation. It is a very efficient and accurate method for computing minimax rational approximants of a continuous real-valued function on subsets of the real line. In Figure \ref{fig:7-exam1-5}, using $m=200,000$ equispaced points in [-1,1] for {\tt d-Lawson}, we plot  maximum errors  of the computed approximants for $|x|$ from {\tt d-Lawson}  and {\tt minimax}. We observed that for types $(n,n)$ with $1\le n\le 30$,   {\tt d-Lawson} provides roughly the same accuracy as {\tt minimax}. For $30<n<60$, it is noticed that the maximum errors  from {\tt d-Lawson} are smaller than {\tt minimax}. This does not imply  that {\tt d-Lawson} is more accurate than {\tt minimax} since {\tt d-Lawson} solves a discrete rational minimax approximation whose optimal maximum error should be smaller than that of the continuous rational minimax approximation. Indeed, {\tt minimax}  matches well with  the theoretical convergence rate \cite{stah:1993} $E_{n,n}(|x|,[-1,1])\sim 8e^{-\pi \sqrt{n}}$. 
For $60<n<80$, rounding errors seem to play a role and {\tt d-Lawson} has difficulty in computing an accurate rational minimax approximant: the final  relative duality gaps  $\epsilon(\bw^{(k)})$ are larger than 0.9.}
   \begin{figure}[h!!!]
\begin{center}
		\includegraphics[width=0.65 \columnwidth,height=0.27 \textheight]{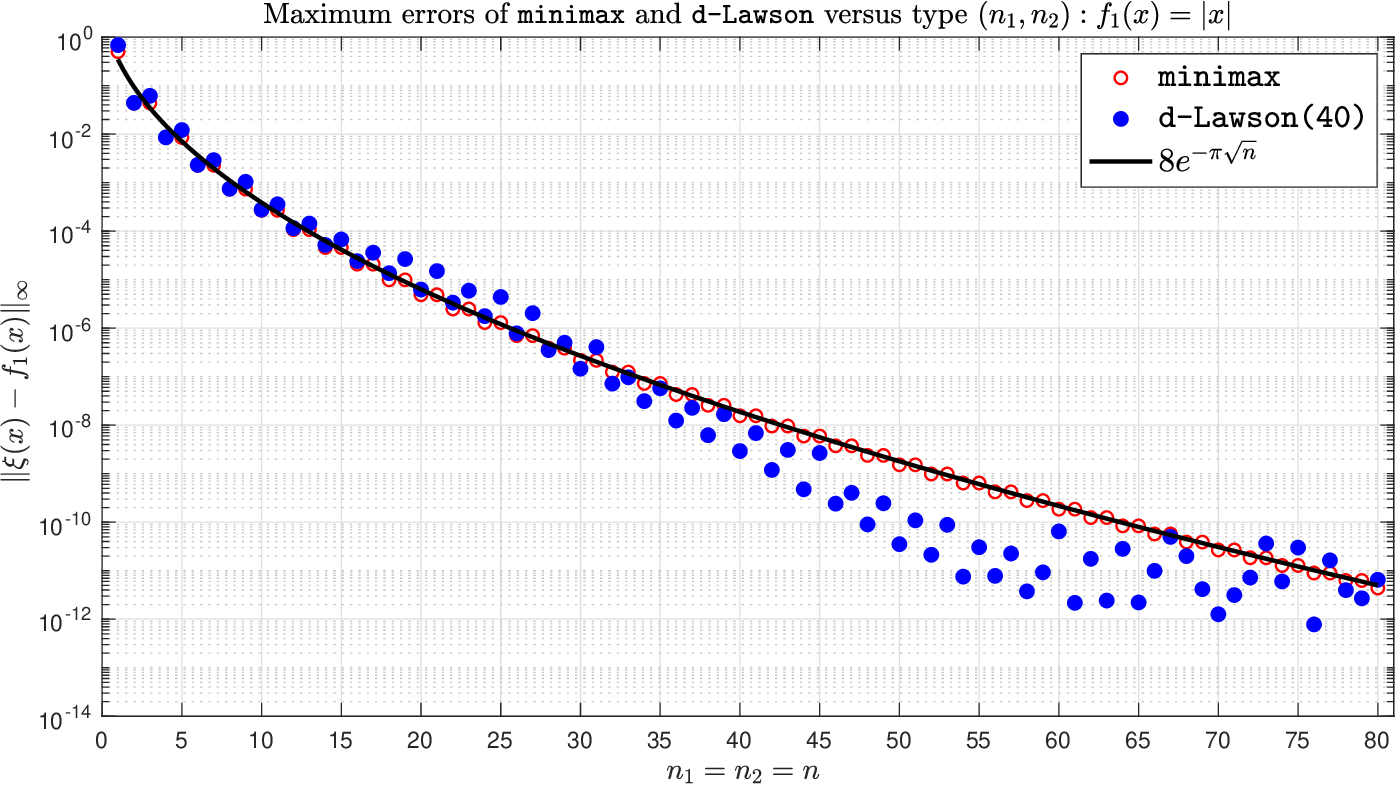}   
\caption{ \small Maximum errors of the computed approximants for $|x|$ from {\tt d-Lawson} (with $m=200,000$ equispaced points in [-1,1]) and {\tt minimax}.  
}
\label{fig:7-exam1-5}
\end{center}
\end{figure}
  
\subsubsection{Function with singularities on an interval}
{We carry out more tests on real functions with singularities in given intervals}. For this purpose, we choose the following two examples used in \cite{fint:2018}:
$$
f_3(x)=-\frac{1}{\log |x|},~~ x \in[-0.1,0.1], ~~{\rm and}~~f_4(x)=\frac{100 \pi\left(x^2-0.36\right)}{\sinh \left(100 \pi\left(x^2-0.36\right)\right)},  ~ x \in[-1,1].
$$
In the left two subfigures of Figure \ref{fig:spike}, $f_3(x)$ and $f_4(x)$ are plotted in the given intervals; the right two subfigures of Figure \ref{fig:spike} present the error curves corresponding to the  types $(7, 11)$ and $(22,20)$, respectively, and the  relative duality gap $\epsilon(\bw^{(k)})$ from {{\tt d-Lawson}  (Algorithm \ref{alg:Lawson})}. In our discrete data, we set $f_3(0)=0$ and $f_4(\pm 0.6)=1$. 
\begin{figure}[h!!!]
\begin{center}
\includegraphics[width=0.85 \linewidth,height=0.3\textheight]{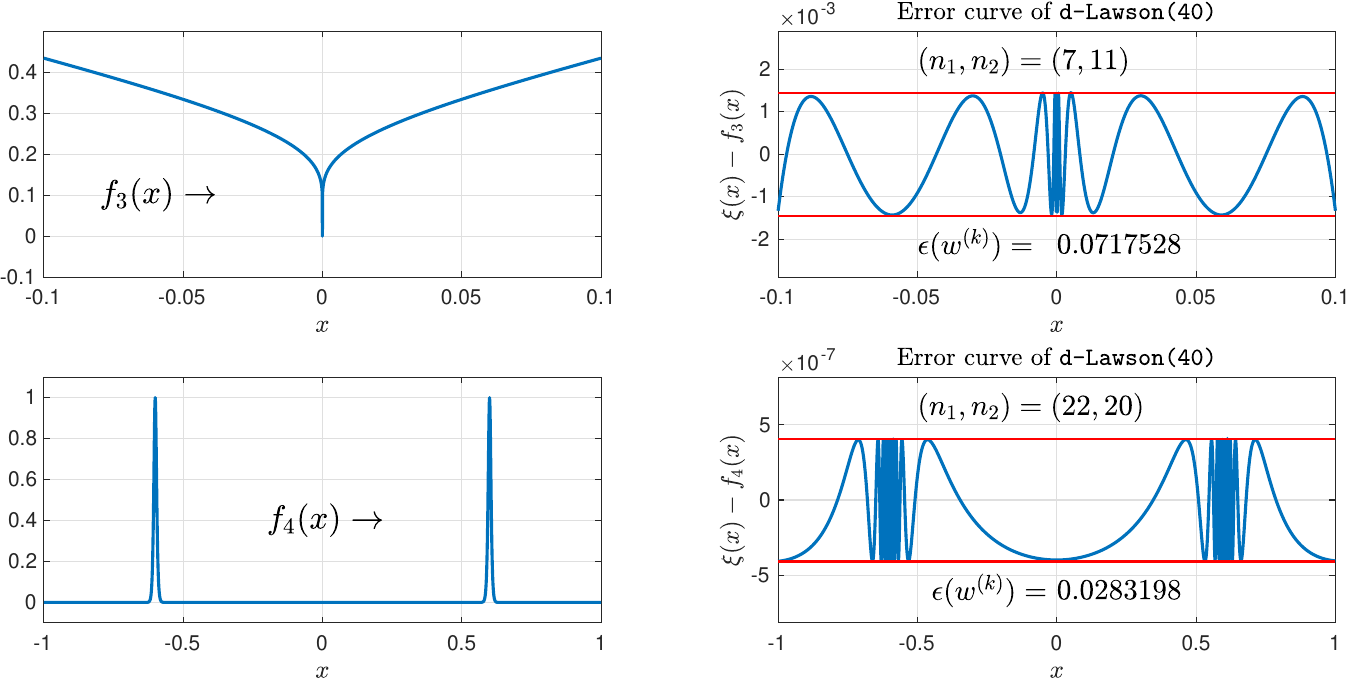} 
\caption{ \small  Left two:  $f_3(x)$ and $f_4(x)$. Right two: error curves and   relative duality gaps $\epsilon(\bw^{(k)})$.} \label{fig:spike}
\end{center}
\end{figure}

To report the results  for $f_3(x)$ and $f_4(x)$, we analogously choose various diagonal types $(n,n)$ and present 
 the maximum errors  $e(\xi)=\|\Bf-\xi(\bx)\|_\infty$ in Figure \ref{fig:7-exam3}. The corresponding relative duality gaps  $\epsilon(\bw^{(k)})$ of {\tt d-Lawson}  and the maximum errors   are given in Tables \ref{tab:7-exam3} and \ref{tab:7-exam4}, respectively. For $f_3(x)$, we noticed that AAA(40) is unstable in some cases but the pure AAA(0) can stably provide good (non-minimax) rational approximants. For $f_4(x)$, the two singularities points $x=\pm 0.6$ certainly make the minimax rational approximation hard, and we noticed that  the methods {(except RKFIT)} perform unstably for some  types, e.g., (14,14), (18,18) and (22,22), of $(n,n)$. Even though it is difficult to check the existence of the minimax rational approximant and Ruttan's sufficient condition for these cases, we can say, from  the relative duality gap $\epsilon(\bw^{(k)})$ in Table \ref{tab:7-exam4}, that these approximants from {{\tt d-Lawson}  (Algorithm \ref{alg:Lawson})} are far from the best.


\begin{figure}[h!!!]
\begin{minipage}[t]{1\textwidth}
 
\begin{center}
{ \hskip-3.45mm
\subfigure{
		\includegraphics[width=0.49\columnwidth,height=0.3\textheight]{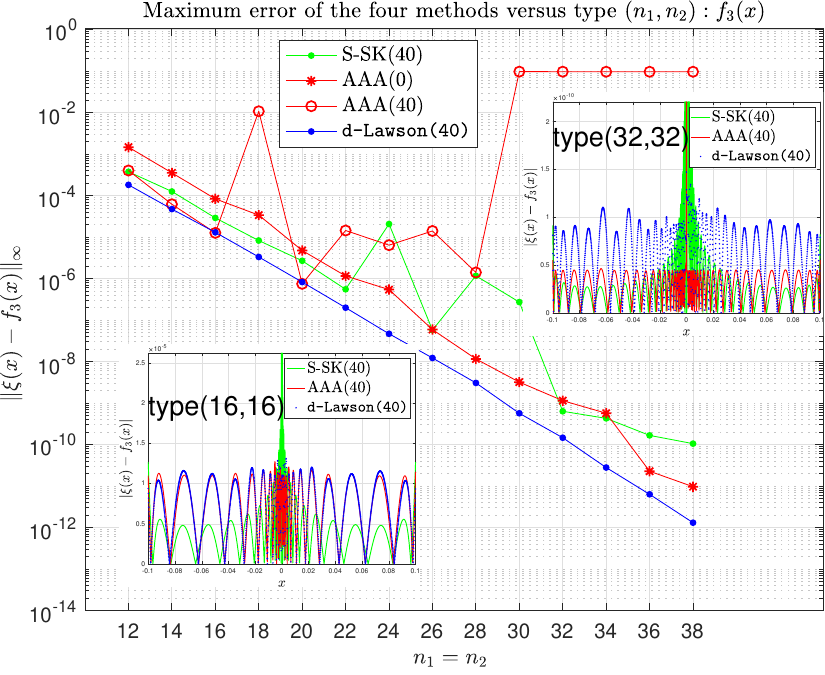}  
	}
	\hspace{0.0cm}\subfigure{
		\includegraphics[width=0.49\columnwidth,height=0.3\textheight]{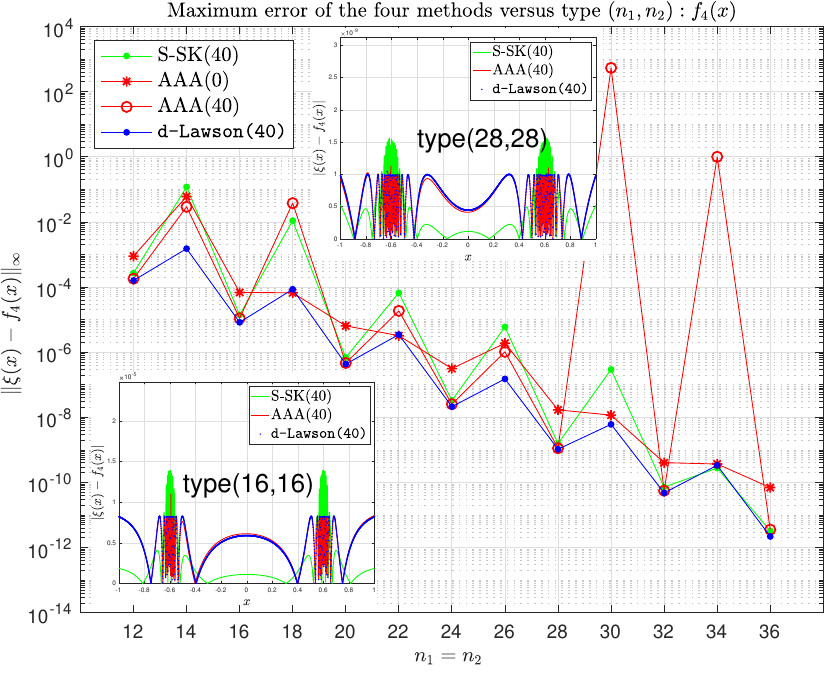}  
	}
  }
\caption{ \small The maximum  error $e(\xi)=\|\Bf-\xi(\bx)\|_\infty$ of
 the four methods with respect to various types $(n,n)$ in approximating $f_3(x)$ (left) and $f_4(x)$ (right). Error curves with particular types $(n,n)$  are also plotted.}\label{fig:7-exam3}
\end{center}
\small
\begin{center} \tabcolsep0.08in
\captionof{table}
{\small The relative duality gap $\epsilon(\bw^{(k)})$ and maximum errors for approximating $f_3(x).$}
\begin{tabular}[h!!!]{|c|c|c|c|c|c|c|}
\hline  
 \tabcolsep0.1in
  & $\epsilon(\bw^{(k)})$ & $\sqrt{d_{2}(\bw^{(k)})}$ &  $ e_{\tt d-Lawson(40)}(\xi) $   &      
  $e_{\rm S-SK(40)}(\xi) $  &   $e_{\rm AAA(40)}(\xi) $   &  ${e_{\rm RKFIT(40)}(\xi) }$ \\  \hline
(12,12) &        0.209038 &      1.4454e-04 &      1.8274e-04 &      3.7847e-04 &      4.0569e-04 &      5.6217e-04 \\ \hline 
(16,16) &        0.242596 &      9.9230e-06 &      1.3101e-05 &      2.9179e-05 &      1.2718e-05 &      4.8388e-05 \\ \hline 
(20,20) &        0.260819 &      6.1323e-07 &      8.2960e-07 &      2.7029e-06 &      7.5248e-07 &      2.9951e-06 \\ \hline 
(24,24) &        0.281158 &      3.3817e-08 &      4.7043e-08 &      2.0787e-05 &      6.4304e-06 &      2.0847e-07 \\ \hline 
(28,28) &        0.481958 &      1.5965e-09 &      3.0819e-09 &      1.1737e-06 &      1.4169e-06 &      1.8890e-06 \\ \hline 
(32,32) &        0.417659 &      8.5690e-11 &      1.4715e-10 &      6.3809e-10 &      9.7529e-02 &      3.5943e-10 \\ \hline 
\end{tabular}
 \label{tab:7-exam3}
\end{center}

\begin{center}\tabcolsep0.08in
\captionof{table}{\small The relative duality gap $\epsilon(\bw^{(k)})$ and maximum errors for approximating $f_4(x).$ }
\begin{tabular}[h!!!]{|c|c|c|c|c|c|c|}
\hline  
 \tabcolsep0.1in
  & $\epsilon(\bw^{(k)})$ & $\sqrt{d_{2}(\bw^{(k)})}$ &  $ e_{\tt d-Lawson(40)}(\xi) $   &      
  $e_{\rm S-SK(40)}(\xi) $  &   $e_{\rm AAA(40)}(\xi) $   &  ${e_{\rm RKFIT(40)}(\xi) }$ \\  \hline
(16,16) &        0.011539 &      8.1768e-06 &      8.2722e-06 &      1.3946e-05 &      1.1018e-05 &      1.3164e-05 \\ \hline 
(18,18) &        0.997610 &      7.9836e-07 &      3.3398e-04 &      1.6991e-04 &      3.7919e-02 &      1.9079e-05 \\ \hline 
(20,20) &        0.027198 &      4.1072e-07 &      4.2221e-07 &      6.8905e-07 &      4.6745e-07 &      7.3126e-07 \\ \hline 
(22,22) &        0.965329 &      7.7170e-08 &      2.2258e-06 &      1.1899e-04 &      1.8589e-05 &      7.3253e-07 \\ \hline 
(24,24) &        0.037679 &      2.0396e-08 &      2.1194e-08 &      3.3124e-08 &      2.5415e-08 &      3.3162e-08 \\ \hline 
(26,26) &        0.944660 &      4.0555e-09 &      7.3284e-08 &      2.5549e-07 &      1.0281e-06 &      3.4716e-08 \\ \hline 
\end{tabular}\label{tab:7-exam4}
\end{center}
\end{minipage}
\end{figure}

\subsection{Numerical experiments on complex cases}\label{subsec-complex}

\subsubsection{Analytic function on a domain $\Omega$}\label{subsubsec-complex-I}
For the complex case, we first test  the methods on the following two basic analytic functions \cite{natr:2020,will:1972} on the unit disc. By the maximum modulus principle for analytic functions, we choose nodes $z_{j} = e^{  -\pi{\rm i} + \frac{{\rm i} 2\pi j}{m} },~j=0,1,\dots,m-1~(m=2000)$ on the unit circle.
\[f_{5}(z) = \tan(z), \quad f_{6}(z) = \log\left(1+\frac{z}{2}\right), \quad  z\in {\cal D}:=\{z\in \bbC: |z|\le 1\}. \] 

By calling the tested methods on $f_5(x)$ and $f_6(x)$ using different diagonal types $(n,n)$, we  present the numerical results in Figure \ref{fig:comp_fun12}, Tables \ref{tab:comp_fun1} and \ref{tab:comp_fun2}. For the two analytic functions on ${\cal D}$, the maximum error reaches the level $10^{-15}$ when the degree $n_1=n_2\ge 9$. It seems that rounding errors play a role so   that the maximum errors do not decrease as the degrees  $n_1=n_2\ge 9$ for all methods. The relative duality gaps $\epsilon(\bw^{(k)})$ from {{\tt d-Lawson}} are also reported in Tables \ref{tab:comp_fun1} and \ref{tab:comp_fun2}. Note that $\epsilon(\bw^{(k)})$ increases significantly from $n_1=n_2\le 7$ to   $n_1=n_2\ge 9$, {possibly due to rounding error effects}.

\begin{figure}[h!!!]
\begin{center}\hskip-3.45mm
{ 
\includegraphics[width=0.49\columnwidth,height=0.3\textheight]{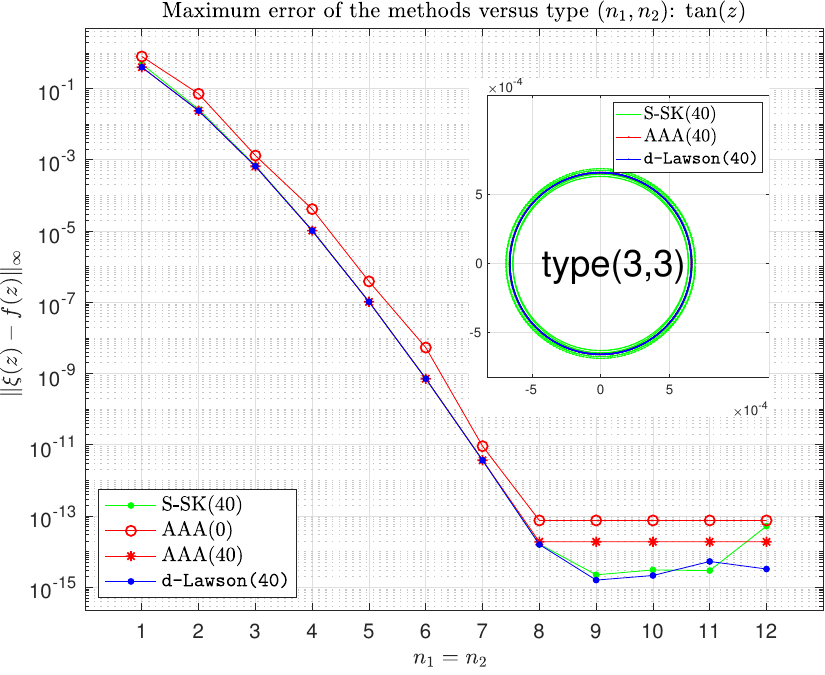} 
\includegraphics[width=0.49\columnwidth,height=0.3\textheight]{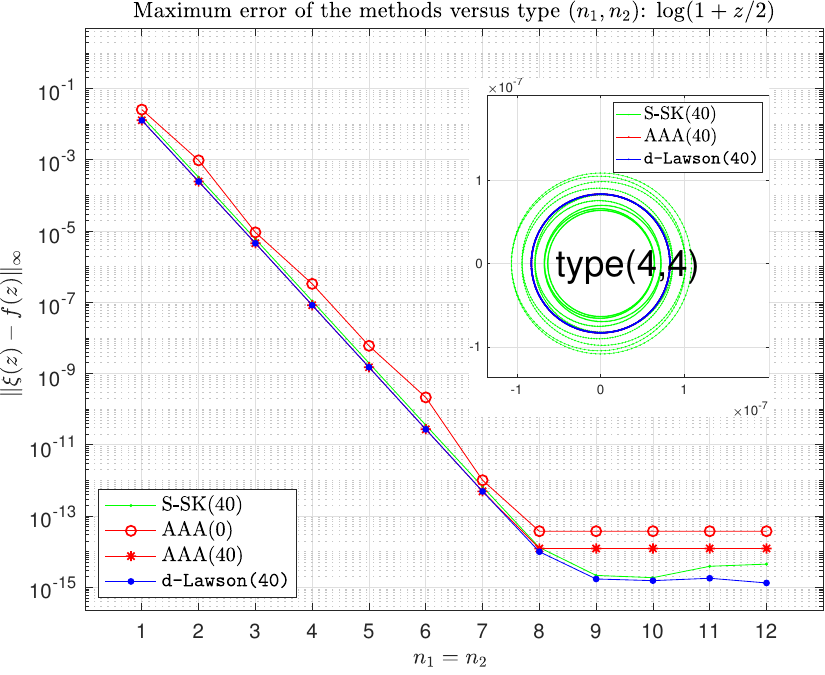}
}
\caption{\small The maximum  error $e(\xi)=\|\Bf-\xi(\bz)\|_\infty$ of
 the four methods with respect to various types $(n,n)$ in approximating $f_5(x)$ (left) and $f_6(x)$ (right). Error curves with particular types $(n,n)$  are also plotted.
}
\label{fig:comp_fun12}
\end{center}
\end{figure}

 \vskip 2mm
 \begin{center}\tabcolsep0.05in
 \captionof{table}{\small The relative duality gap $\epsilon(\bw^{(k)})$ and maximum errors for approximating $f_5(x).$ } 
\begin{tabular}[h!!!]{|c|c|c|c|c|c|c|}
\hline  
 \tabcolsep0.01in
  & $\epsilon(\bw^{(k)})$ & $\sqrt{d_{2}(\bw^{(k)})}$ &  $ e_{\tt d-Lawson(40)}(\xi) $   &      
  $e_{\rm S-SK(40)}(\xi) $  &   $e_{\rm AAA(40)}(\xi) $   &  ${e_{\rm RKFIT(40)}(\xi) }$ \\  \hline
(1,1) &        0.001685 &      3.9727e-01 &      3.9794e-01 &      4.8903e-01 &      3.9806e-01 &      5.5741e-01 \\ \hline 
(3,3) &        0.000024 &      6.5927e-04 &      6.5929e-04 &      6.8943e-04 &      6.5929e-04 &      6.8942e-04 \\ \hline 
(5,5) &        0.000009 &      1.0339e-07 &      1.0339e-07 &      1.0517e-07 &      1.0339e-07 &      1.0517e-07 \\ \hline 
(7,7) &        0.000354 &      3.6816e-12 &      3.6829e-12 &      3.7164e-12 &      3.6847e-12 &      3.7173e-12 \\ \hline 
(9,9) &        0.570238 &      6.8854e-16 &      1.6021e-15 &      2.2542e-15 &      1.9033e-14 &      1.6012e-15 \\ \hline 
(11,11) &        0.289592 &      3.7932e-15 &      5.3395e-15 &      2.9624e-15 &      1.9033e-14 &      5.2659e-15 \\ \hline 
\end{tabular}\label{tab:7-exam5} 
\label{tab:comp_fun1}
\end{center}
 
 \begin{center}\tabcolsep0.05in
 \captionof{table}{\small The relative duality gap $\epsilon(\bw^{(k)})$ and maximum errors for approximating $f_6(x).$ }
\begin{tabular}[h!!!]{|c|c|c|c|c|c|c|}
\hline  
 \tabcolsep0.01in
  & $\epsilon(\bw^{(k)})$ & $\sqrt{d_{2}(\bw^{(k)})}$ &  $ e_{\tt d-Lawson(40)}(\xi) $   &      
  $e_{\rm S-SK(40)}(\xi) $  &   $e_{\rm AAA(40)}(\xi) $   &  ${e_{\rm RKFIT(40)}(\xi) }$ \\  \hline
(1,1) &        0.000354 &      1.2849e-02 &      1.2854e-02 &      1.6878e-02 &      1.2854e-02 &      1.6944e-02 \\ \hline 
(3,3) &        0.000008 &      4.5539e-06 &      4.5539e-06 &      5.9751e-06 &      4.5539e-06 &      5.9714e-06 \\ \hline 
(5,5) &        0.000006 &      1.5094e-09 &      1.5094e-09 &      1.9795e-09 &      1.5094e-09 &      1.9792e-09 \\ \hline 
(7,7) &        0.002465 &      4.9347e-13 &      4.9469e-13 &      6.7134e-13 &      4.9556e-13 &      6.4635e-13 \\ \hline 
(9,9) &        0.529233 &      8.1036e-16 &      1.7213e-15 &      2.1658e-15 &      1.2369e-14 &      1.7919e-15 \\ \hline 
(11,11) &        0.605217 &      7.1156e-16 &      1.8024e-15 &      3.9083e-15 &      1.2369e-14 &      1.2325e-15 \\ \hline 
\end{tabular}\label{tab:comp_fun2}
\end{center}

\subsubsection{Function with singularities on a domain $\Omega$}\label{subsubsec-complex-II}
As our final part of numerical testing, we choose two functions $f_7(z)$ and $f_8(z)$ with singularities on the unit disc ${\cal D}$.  In particular, $f_7(z)$ \cite{elwi:1976} has a pole $z=-\frac12$ and  $f_8(z)$ \cite[Fig 4.3, Sec.4]{natr:2020} has  branch point singularities $z=e^{\frac{{\tt i}\pi}{4}}, e^{\frac{{\tt i}3\pi}{4}}, e^{\frac{{\tt i}5\pi}{4}}, e^{\frac{{\tt i}7\pi}{4}}$ on the unit disc $ {\cal D}$. 
\begin{align*}
& f_{7}(z) = (1+2z)^{-1/2} ~~{\rm with ~nodes}\quad z_{j} = e^{  \frac{-\pi{\rm i} }{2}+ \frac{j\pi{\rm i} }{m} }, \quad
j =0,1,\dots,  { m=2000, } \\
& f_{8}(z) = (1+z^{4})^{1/2}~~{\rm with ~nodes}\quad z_{j} = e^{{\rm i} \pi/4 \tanh\left(-12 + \frac{ 24j}{m}\right)},\quad 
j =0,1,\dots,  { m=2000. } 
\end{align*}

The maximum errors given in Figure \ref{fig:comp_fun34} again demonstrate the difficulty of computing the   rational minimax approximation whenever the maximum error is near or below the machine precision for the four methods. In particular, for $f_7(z)$, the computed approximants of degrees  $n_1=n_2\ge 9$ from the four methods are not the   minimax solutions of \eqref{eq:bestf} by checking the left of Figure \ref{fig:comp_fun34}  and  $\epsilon(\bw^{(k)})$ in  Table \ref{tab:7-exam7}. Similarly for $f_8(z)$, by investigating $\epsilon(\bw^{(k)})$ in Table \ref{tab:7-exam8}, we know that the computed solutions with degrees $n_1=n_2>14$ {are only {near best}, but are also} good rational approximants. For {\tt d-Lawson}, increasing the number of iterations $k$ can improve the accuracy with smaller gaps  $\epsilon(\bw^{(k)})$ (for instance, gaps $\epsilon(\bw^{(k)})$ for  degrees $n_1=n_2 \ge 14$  get to the level of $10^{-2}$ when $k=100$ for $f_8(z)$). Overall,  {{\tt d-Lawson}  (Algorithm \ref{alg:Lawson})} turns out to be an effective approach for  the rational minimax problem \eqref{eq:bestf}.
 
\begin{figure}[h!!!]
\begin{minipage}[t]{1\textwidth}
\begin{center}
{ \hskip-3.45mm
\includegraphics[width=0.49\columnwidth,height=0.3\textheight]{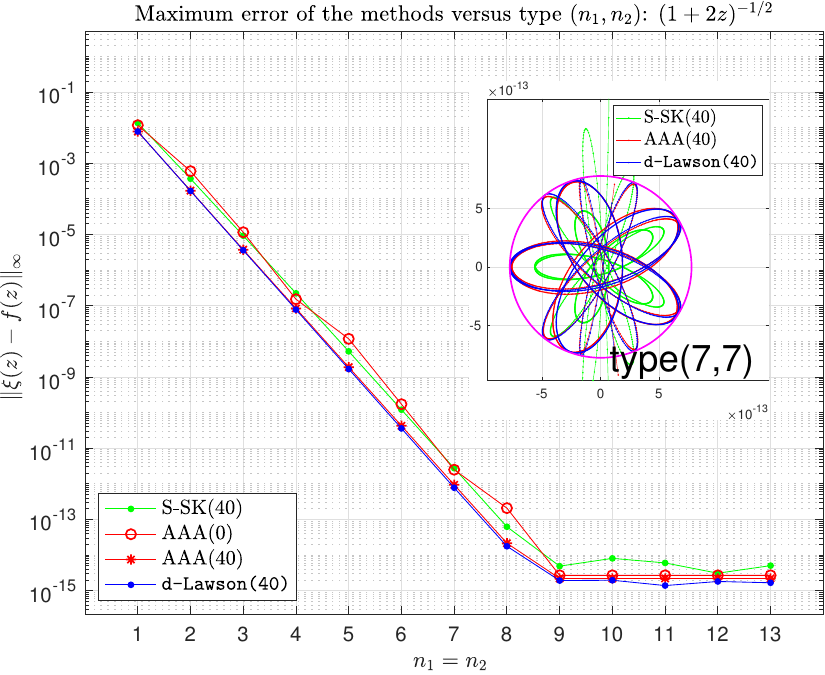} 
\includegraphics[width=0.49\columnwidth,height=0.3\textheight]{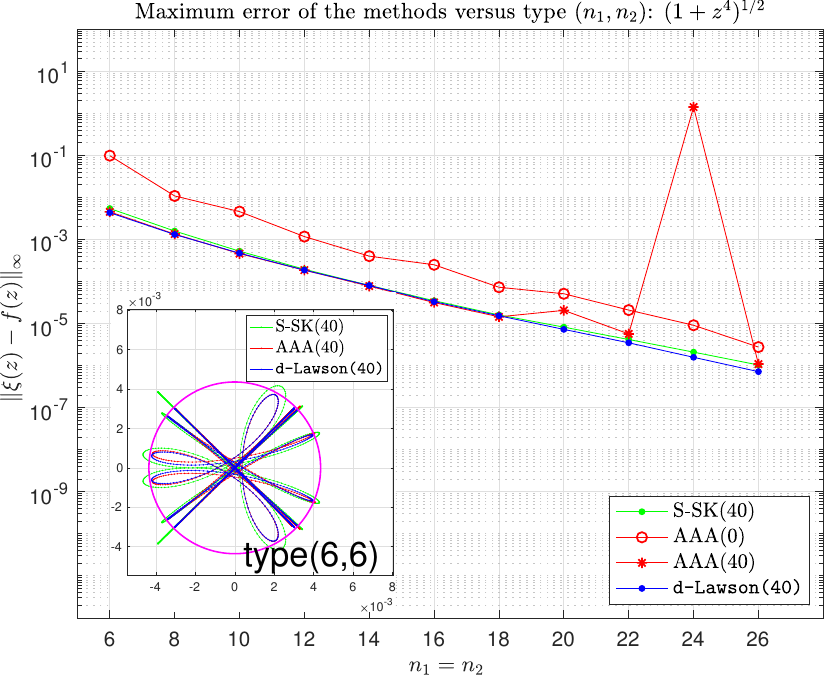}
}
\caption{\small The maximum  error $e(\xi)=\|\Bf-\xi(\bz)\|_\infty$ of
 the four methods with respect to various types $(n,n)$ in approximating $f_7(x)$ (left) and $f_8(x)$ (right). Error curves with particular types $(n,n)$  and the maximum circle  (in magenta) $\{z\in \bbC:|z|\le e(\xi^{(k)})\}$ with the radius  $e(\xi^{(k)})$ computed from {\tt d-Lawson}(40) are also plotted.}
\label{fig:comp_fun34}
\end{center}

\begin{center}\tabcolsep0.05in
\captionof{table}{\small The relative duality gap $\epsilon(\bw^{(k)})$ and maximum errors for approximating  $f_7(x).$ }
\begin{tabular}[h!!!]{|c|c|c|c|c|c|c|}
\hline  
\tabcolsep0.01in
  & $\epsilon(\bw^{(k)})$ & $\sqrt{d_{2}(\bw^{(k)})}$ &  $ e_{\tt d-Lawson(40)}(\xi) $   &      
  $e_{\rm S-SK(40)}(\xi) $  &   $e_{\rm AAA(40)}(\xi) $   &  ${e_{\rm RKFIT(40)}(\xi) }$ \\  \hline
(1,1) &        0.020223 &      7.6419e-03 &      7.7996e-03 &      1.3523e-02 &      7.7565e-03 &      1.3517e-02 \\ \hline 
(3,3) &        0.014089 &      3.5598e-06 &      3.6107e-06 &      9.2676e-06 &      3.7489e-06 &      9.2674e-06 \\ \hline 
(5,5) &        0.014708 &      1.6529e-09 &      1.6775e-09 &      5.2525e-09 &      1.8907e-09 &      5.2524e-09 \\ \hline 
(7,7) &        0.017286 &      7.6671e-13 &      7.8020e-13 &      2.7591e-12 &      9.4454e-13 &      2.7525e-12 \\ \hline 
(9,9) &        0.590636 &      7.9275e-16 &      1.9365e-15 &      4.9164e-15 &      2.2220e-15 &      3.8984e-15 \\ \hline 
(11,11) &        0.501714 &      6.9812e-16 &      1.4010e-15 &      6.0568e-15 &      2.2220e-15 &      4.0665e-15 \\ \hline 
\end{tabular}\label{tab:7-exam7}
\end{center}

\vskip 5mm
\begin{center}\tabcolsep0.05in
\captionof{table}{\small The relative duality gap $\epsilon(\bw^{(k)})$ and maximum errors for approximating  $f_8(x).$ }
\begin{tabular}[h!!!]{|c|c|c|c|c|c|c|}
\hline  
\tabcolsep0.01in
  & $\epsilon(\bw^{(k)})$ & $\sqrt{d_{2}(\bw^{(k)})}$ &  $ e_{\tt d-Lawson(40)}(\xi) $   &      
  $e_{\rm S-SK(40)}(\xi) $  &   $e_{\rm AAA(40)}(\xi) $   &  ${e_{\rm RKFIT(40)}(\xi) }$ \\  \hline
(6,6) &        0.036948 &      4.1942e-03 &      4.3551e-03 &      5.4714e-03 &      4.5470e-03 &      5.7982e-03 \\ \hline 
(10,10) &        0.095931 &      4.2888e-04 &      4.7439e-04 &      5.2787e-04 &      4.6216e-04 &      5.6462e-04 \\ \hline 
(14,14) &        0.195433 &      6.4913e-05 &      8.0681e-05 &      8.1036e-05 &      7.8097e-05 &      7.9326e-05 \\ \hline 
(18,18) &        0.232264 &      1.1807e-05 &      1.5379e-05 &      1.6084e-05 &      1.4523e-05 &      1.7184e-05 \\ \hline 
(22,22) &        0.349815 &      2.2758e-06 &      3.5002e-06 &      4.1924e-06 &      5.7008e-06 &      4.5689e-06 \\ \hline 
(26,26) &        0.358080 &      4.6298e-07 &      7.2124e-07 &      1.0385e-06 &      1.0908e-06 &      1.1394e-06 \\ \hline 
\end{tabular}\label{tab:7-exam8}
\end{center}
\end{minipage}
\end{figure}

We conclude this section by emphasizing again the properties of the monotonic convergence, the weak and strong duality related with {{\tt d-Lawson}  (Algorithm \ref{alg:Lawson})} in the complex situation. These properties are closely connected to the underlying dual {problem} \eqref{eq:dualPX} and provide theoretical foundations for the new Lawson's iteration.  In Figure \ref{fig:comp_fun1}, we plot   the sequences $\left\{e(\xi^{(k)})\right\}$  and $\left\{\sqrt{d_2(\bw^{(k)})}\right\}$  for $f_5(z)$ and $f_7(z)$ with type ${(3,7)}$ and $(5,5)$, respectively. Error curves $\xi^{(k)}(\bz)-\Bf$  in the complex plane after certain numbers of iterations are also included. We observed that the weight $\bw^{(k)}$ computed by {\tt d-Lawson} verifies strong duality \eqref{eq:strongdualityRuttan}, implying Ruttan's sufficient condition holds and {\tt d-Lawson} successfully solves globally the dual {problem} \eqref{eq:dualPX}.

\begin{figure}[h!!!]
\begin{center}
{
\includegraphics[width=0.47\columnwidth,height=0.3\textheight]{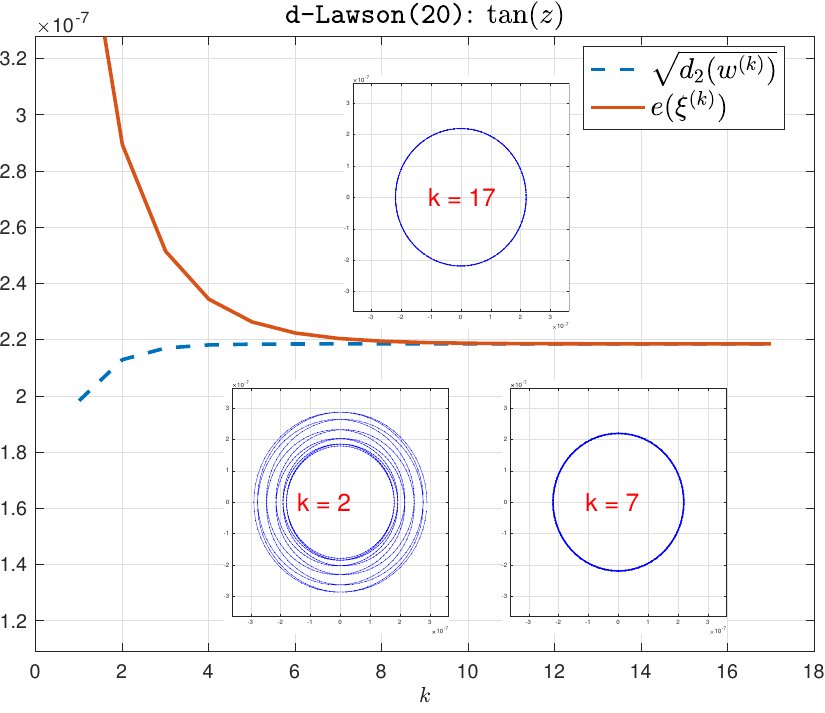} \hskip 4mm
\includegraphics[width=0.47\columnwidth,height=0.3\textheight]{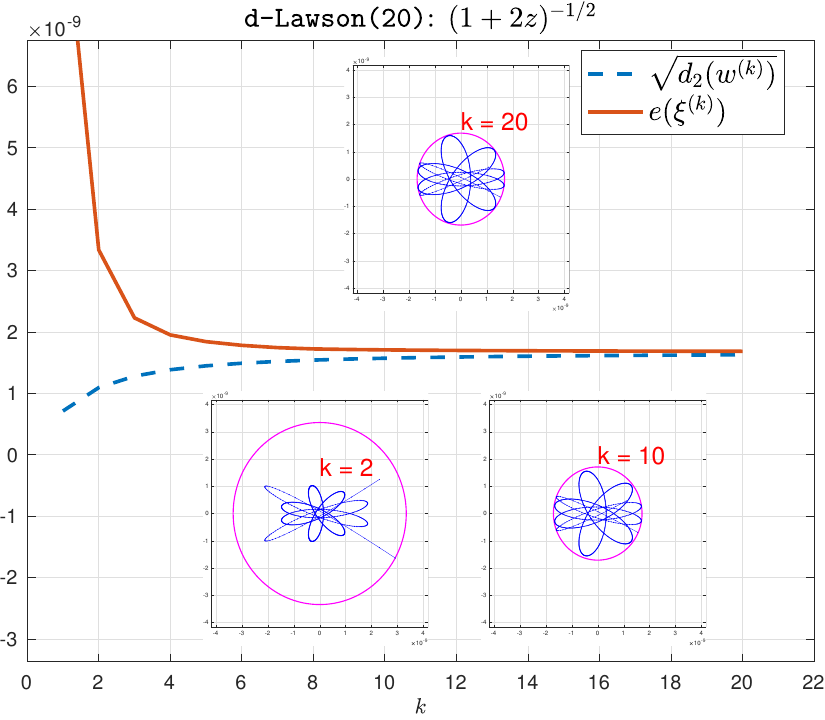}
}
\caption{\small The sequences $\left\{e(\xi^{(k)})\right\}$  and $\left\{\sqrt{d_2(\bw^{(k)})}\right\}$ and error curves $\xi^{(k)}(\bz)-\Bf$  in the complex plane from Algorithm \ref{alg:Lawson} for $f_5(z)$ (left)  with type (3,7), and for $f_7(z)$ (right)  with type (5,5).}\vskip -2mm
\label{fig:comp_fun1}
\end{center}
\end{figure}

\section{Concluding remarks and further issues}\label{sec_conclude} 
{In this paper, we have constructed a convex optimization problem \eqref{eq:dualPX}   to  solve the discrete rational minimax approximation \eqref{eq:bestf}, and designed  a  new version of Lawson's iteration {\tt d-Lawson} (Algorithm \ref{alg:Lawson})}. The convex {optimization} \eqref{eq:dualPX} is derived as a dual {problem} of a certain linearized reformulation of the  rational minimax approximation. We showed that there is no duality gap between the primal and the dual under Ruttan's sufficient condition. In practice, Ruttan's sufficient condition, or equivalently, strong duality, is checkable through \eqref{eq:strongdualityRuttan} after solving the dual {problem} \eqref{eq:dualPX} globally. The  relative duality gap $\epsilon(\bw^{(k)})$ also reflects the accuracy of the minimax approximant. Our numerical experiments  indicate that the easily-implementable Lawson's iteration (Algorithm \ref{alg:Lawson}) is an effective approach for solving the convex programming \eqref{eq:dualPX}.

Even {if} the dual {problem} \eqref{eq:dualPX} can provide a new perspective for {computing} the rational minimax approximation \eqref{eq:bestf}, several important issues have not been discussed in this paper. Indeed, the observed linear and monotonic convergence in our numerical experiments of the new Lawson's iteration has not been established in theory.  {It was pointed out in \cite{natr:2020} for the AAA-Lawson algorithm}   {\it `it is a challenge for the future to develop further improvements to the algorithm that might ensure its convergence in all circumstances and to support these with theoretical guarantees'}. {We} think an associated  dual problem  can play a role {in analyzing} the theoretical convergence of the corresponding  Lawson's iteration. To improve the stability of {{\tt d-Lawson}  (Algorithm \ref{alg:Lawson})}, barycentric representations, instead of the Vandermonde basis,  may also be employed for solving dual {problem} \eqref{eq:dualPX}. Moreover, the present new Lawson's iteration can also fail to solve the dual \eqref{eq:dualPX} globally (i.e., the duality gap $\epsilon(\bw^{(k)})$ does not converge to zero), and even {if} it converges, the convergence in some cases can be slow  and  less efficient than the AAA, AAA-Lawson algorithm. Fast and effective methods for globally solving the dual {problem} are certainly desired as detecting the reference points (i.e., nodes in ${\cal X}_e(\xi^*)$ of \eqref{eq:extremalset}) corresponds to finding the nonzero weights $w_j^*$ at the maximizer $\bw^*$ of the dual {problem} \eqref{eq:dualPX};  reference points are  helpful in   the representation of $\xi^*=p^*/q^*$ either by finding roots \cite{itna:2018} of $p^*$ and $q^*$, or their associated  coefficient vectors in certain bases. As the dual \eqref{eq:dualPX}  is to maximize a concave function $d_2(\bw)$ over   the probability simplex  ${\cal S}$ and  both the gradient and the Hessian of $d_2(\bw)$ are provided in section \ref{sec_gradHess}, other methods, including the projected gradient iteration, the interior-point method \cite[Chapter 19]{nowr:2006} or  methods for convex optimization \cite{nest:2018}, can be applied for solving \eqref{eq:dualPX}.  All these issues deserve further {investigation}.

\appendix
\section{Appendix}\label{AppendixA}
We provide the proofs for Theorem \ref{thm:linearity},  Proposition \ref{prop:grad} and Corollary \ref{cor:Hess2} in this appendix. 
\vskip 2mm

\noindent {\it Proof of  Theorem \ref{thm:linearity}.}
The result is trivial if $\eta_\infty=0$. For $\eta_\infty>0$, it suffices to prove the result for $\alpha=1$ as $|f_jq(x_j)-p(x_j)|^\alpha\le \eta |q(x_j)|^\alpha$ holds if and only if $|f_jq(x_j)-p(x_j)| \le \eta^\frac{1}{\alpha} |q(x_j)|$. 

Denote by $\wtd\eta$ the infimum of \eqref{eq:linearity} and $\wtd\eta\le \eta_\infty$ is obvious as the triple $(\eta_\infty,p^*,q^*)$ is feasible for  \eqref{eq:linearity}. Suppose $\wtd\eta<\eta_\infty$. Then by definition of the infimum, for $\eta_\epsilon=\wtd \eta+\epsilon$ with $0<\epsilon=\frac{\eta_\infty-\wtd\eta}{2}$, there is a pair $(p_\epsilon,q_\epsilon)$ with $p_\epsilon\in \bbP_{n_1}$ and $q_\epsilon\in \bbP_{n_2}\setminus\{0\}$ such that 
$$
 |f_jq_\epsilon(x_j)-p_\epsilon(x_j)| \le \eta_\epsilon |q_\epsilon(x_j)|, \quad \forall j\in [m].
$$  
The proof is finished if $|q_\epsilon(x_j)|\ne 0$   $\forall j\in [m]$ because the rational {approximant} $\xi_{\epsilon}(x)=p_\epsilon(x)/q_\epsilon(x)\in \scrR_{(n_1,n_2)}$ gives $\|\Bf-\xi_{\epsilon}(\bx)\|_\infty\le \eta_\epsilon<\eta_\infty.$
Thus, we assume $\emptyset \ne {\cal J}_\epsilon:=\{j\in[m]|q_{\epsilon}(x_j)=0\}\subset[m]$. Since $q_\epsilon\not \equiv 0$, the cardinality $|{\cal J}_\epsilon|$ satisfies $1\le |{\cal J}_\epsilon|\le n_2$ and  $[m]\setminus{\cal J}_\epsilon\ne  \emptyset$. By $
 |f_jq_\epsilon(x_j)-p_\epsilon(x_j)| \le \eta_\epsilon |q_\epsilon(x_j)| 
$, it is true that $p_{\epsilon}(x_j)=0$ for $j\in {\cal J}_\epsilon$.
Now for any $\delta>0$, we define $$p_\delta(x)=p_\epsilon(x)+\delta p^*(x)\in \bbP_{n_1} \quad {\rm and} \quad q_\delta(x)=q_\epsilon(x)+\delta q^*(x)\in \bbP_{n_2}.$$ Assume   $\delta>0$ is sufficiently small so that $q_\delta(x_j)\ne 0$ for all $j\in [m]$.  We shall show that for any sufficiently small $\delta>0$, the rational {approximant} $\xi_{\delta}(x)=p_\delta(x)/q_\delta(x)$ satisfies 
$
\|\Bf-\xi_{\delta}(\bx)\|_\infty\le \eta_\infty
$
and thus $\xi_{\delta} \ne \xi_* $ is another solution to \eqref{eq:bestf}. To this end, we  consider the cases $j\in  {\cal J}_\epsilon$ and $j\in  [m]\setminus{\cal J}_\epsilon$. 

For $j\in  {\cal J}_\epsilon$, we have $$
|f_jq_\delta(x_j)-p_\delta(x_j)|=\delta|f_jq^*(x_j)-p^*(x_j)|\le \eta_\infty  \delta |q^*(x_j)|=\eta_\infty   |q_\delta(x_j)|.
$$

For $j\in  [m]\setminus{\cal J}_\epsilon$, note that 
\begin{align*}
|f_jq_\delta(x_j)-p_\delta(x_j)|&\le \delta|f_jq^*(x_j)-p^*(x_j)|+ |f_jq_\epsilon(x_j)-p_\epsilon(x_j)| \le \eta_\infty  \delta |q^*(x_j)|+\eta_\epsilon   |q_\epsilon(x_j)|.
\end{align*}
Let $$c_0=\max_{j\in [m]\setminus{\cal J}_\epsilon}|q^*(x_j)|>0 \quad {\rm and}\quad c_1=\min_{j\in [m]\setminus{\cal J}_\epsilon}|q_\epsilon(x_j)|>0$$ and choose $0<\delta<\frac{c_1\epsilon}{4c_0\eta_\infty}$. We have
\begin{align*}
\left(\eta_\epsilon+\frac{\epsilon}{2}\right)|q_\delta(x_j)|&\ge \left(\eta_\epsilon+\frac{\epsilon}{2}\right)|q_\epsilon(x_j)|- \left(\eta_\epsilon+\frac{\epsilon}{2}\right) \delta |q^*(x_j)|\\
&> \left(\eta_\epsilon+\frac{\epsilon}{2}\right)|q_\epsilon(x_j)|-\eta_\infty\delta c_0\\
&\ge  \eta_\epsilon|q_\epsilon(x_j)| +\frac{\epsilon}{2} c_1-\eta_\infty\delta c_0\\
&\ge \eta_\epsilon|q_\epsilon(x_j)| +\eta_\infty  \delta c_0\\
&\ge \eta_\epsilon   |q_\epsilon(x_j)|+\eta_\infty  \delta |q^*(x_j)|;
\end{align*}
therefore, for any $j\in  [m]\setminus{\cal J}_\epsilon$, it holds that
$$
|f_jq_\delta(x_j)-p_\delta(x_j)|\le \left(\eta_\epsilon+\frac{\epsilon}{2}\right)|q_\delta(x_j)|<\eta_\infty|q_\delta(x_j)|,
$$
which, with $j\in  {\cal J}_\epsilon$, shows that $\xi_\delta=p_\delta/q_\delta$ satisfies $\|\xi_\delta(\bx)-\Bf\|_\infty=\eta_\infty$ and is a solution to \eqref{eq:bestf} for any sufficiently small $\delta>0$. Finally, we show that $\xi_\delta\not\equiv\xi^*$ because if $\xi_\delta\equiv\xi^*$, i.e., $p_\epsilon q^*\equiv p^* q_\epsilon$, we have $p^*/q^*\equiv \wtd p_\epsilon/\wtd q_\epsilon$, where $\wtd p_\epsilon/\wtd q_\epsilon$ is the irreducible representation of $  p_\epsilon/  q_\epsilon$. Since $p_\epsilon$ and $q_\epsilon$ have at least common zeros of $x_j$ for $j\in {\cal J}_\epsilon$ but $p^*/q^*$ is non-degenerate, it contradicts with $p^*/q^*\equiv \wtd p_\epsilon/\wtd q_\epsilon$. Thus  $\eta_\infty=\wtd \eta$ and the conclusion follows. $\hfill\square$

\vskip 2mm

\noindent {\it Proof of  Proposition \ref{prop:grad}.}
(i)  For $\bw>0$, we know from Proposition \ref{prop:dual_GEP} (iii) that $\wtd n_1=n_1+1$ and $\wtd n_2=n_2+1$, and $d_2(\bw)$ is a simple eigenvalue of a positive semi-definite Hermitian matrix and thus is differentiable at $\bw$ (see e.g.,  \cite[Section 7.2.2]{govl:2013}). {We obtain the formulation of the gradient $\nabla d_2(\bw)$ via computing the directional  derivative along a given  $\Delta_w\in \bbR^m$. To this end, let $\bw(t)= \bw+t \Delta_w$ and $\bc(\bw(t))$ be the corresponding eigenvector in \eqref{eq:Heig} for any sufficiently small $t$}. We have 
$$
\left(A_{\bw(t)}-d_2(\bw(t))B_{\bw(t)})\bc(\bw(t)\right)=0,~~\bc(\bw(t))^{\HH}B_{\bw(t)}\bc(\bw(t))=1.
$$
Differentiate the first term with respect to $t$ and note  $\dot{A}_{\bw(0)}=[-\Psi,F\Phi]^{\HH}\diag(\Delta_w)[-\Psi,F\Phi]$ and $\dot{B}_{\bw(0)}=[0,\Phi]^{\HH}\diag(\Delta_w)[0,\Phi]$ to have  
\begin{align*}
\left(\dot{A}_{\bw(0)}-\dot{d}_2(\bw(0))B_{\bw(0)}-d_2(\bw(0))\dot{B}_{\bw(0)}\right)\bc(\bw(0))+(A_{\bw}-d_2(\bw)B_{\bw})\dot{\bc}(\bw(0))=0;
\end{align*}
Pre-multiplying $\bc(\bw)^{\HH}$ on both sides and using $$\bc(\bw)^{\HH}(A_{\bw}-d_2(\bw)B_{\bw})=0 \quad {\rm and}\quad \bc(\bw(t))^{\HH}B_{\bw(t)}\bc(\bw(t))=1$$ yield 
\begin{align}\nonumber
\dot{d}_2(\bw(0))   &=\nabla d_2(\bw)^{\T}\Delta_w = \bc^{\HH}(\bw(0))\dot{A}_{\bw(0)}\bc(\bw(0))-d_2(\bw(0))\bc^{\HH}(\bw(0))\dot{B}_{\bw(0)}\bc(\bw(0))\\\nonumber
 &=\left(|F \bq -\bp |^2-d_2(\bw)|\bq |^2\right)^{\T}\Delta_w,
 \end{align}
where $|F \bq -\bp |^2-d_2(\bw)|\bq |^2$ is defined in \eqref{eq:gradd2}. Thus (i) follows.

(ii) {To obtain the formulation of the Hessian, we first compute its $(j,k)$th entry.  Note}   
\begin{align} \nonumber
 [\nabla^2 d_2(\bw)]_{jk} 
=&\frac{\partial [\nabla  d_2(\bw)]_j}{\partial w_k}=\be_j^{\T}\nabla^2 d_2(\bw)\be_k\\\nonumber
=&2{\rm Re}\left(\overline{(f_j q_j- p_j)}\frac{\partial(f_j q_j- p_j)}{\partial w_k}-d_2(\bw)\overline{ q_j}\frac{\partial  q_j}{\partial w_k}\right)- \frac{\partial d_2(\bw)}{\partial w_k} | q_j|^2\\ 
\label{eq:hess1}
=&2\be_j^{\T}{\rm Re}\left(\left( {R}_1^{\HH} E_1 -d_2(\bw) {R}_2^{\HH}E_2\right)\frac{\partial \bc(\bw)}{\partial w_k}\right)- [\nabla  d_2(\bw)]_k|  q_j|^2. 
\end{align}
Now, for $\frac{\partial \bc(\bw)}{\partial w_k}$,  we can compute the partial derivative with respect to $w_k$ of the first equation in \eqref{eq:dual_GEP} to have
\begin{align}\nonumber
0&=\left(E_1^{\HH}\be_k\be_k^{\T}E_1-[\nabla d_2(\bw)]_k B_{\bw}-d_2(\bw)E_2^{\HH}\be_k\be_k^{\T}E_2\right)\bc(\bw)+\left(A_{\bw}-d_2(\bw)B_{\bw}\right)\frac{\partial \bc(\bw)}{\partial w_k}\\\nonumber
&=(E_1^{\HH}R_1 - B_{\bw}\bc(\bw)(\nabla d_2(\bw))^{\T}-d_2(\bw)E_2^{\HH}R_2 )\be_k+(A_{\bw}-d_2(\bw)B_{\bw})\frac{\partial \bc(\bw)}{\partial w_k}\\\label{eq:partialcwk}
&=R_3^{\HH}\be_k+(A_{\bw}-d_2(\bw)B_{\bw})\frac{\partial \bc(\bw)}{\partial w_k},
\end{align} 
i.e., $
\left(A_{\bw}-d_2(\bw)B_{\bw}\right)\frac{\partial \bc(\bw)}{\partial w_k}=-R_3\be_k.
$
Note that $\bc(\bw)^{\HH}R_3^{\HH}\be_k=0$ and $(A_{\bw}-d_2(\bw)B_{\bw})\bc(\bw)=0$; hence $\frac{\partial \bc(\bw)}{\partial w_k}$ can be expressed as
\begin{equation}\label{eq:hess2}
\frac{\partial \bc(\bw)}{\partial w_k}=-(A_{\bw}-d_2(\bw)B_{\bw})^{\dag}R_3^{\HH}\be_k+\iota_k \bc(\bw),~~{\rm for~some ~}\iota_k\in \bbC.
\end{equation}
To obtain $\iota_k$, from $\frac{\bc(\bw)^{\HH}B_{\bw} \bc(\bw)}{\partial w_k}=0$, $\bc(\bw)^{\HH}B_{\bw}\bc(\bw)=1$ and  \eqref{eq:hess2}, it follows 
\begin{align*}
0&=2{\rm Re}\left(\bc(\bw)^{\HH}B_w \frac{\partial \bc(\bw)}{\partial w_k}\right)+\bc(\bw)^{\HH}E_2^{\T}\be_k\be_k^{\T}E_2\bc(\bw)\\
&=2{\rm Re}\left(\bc(\bw)^{\HH}B_w \frac{\partial \bc(\bw)}{\partial w_k}\right)+| q_k|^2\\
&=-2{\rm Re}\left(\bc(\bw)^{\HH}B_w (A_{\bw}-d_2(\bw)B_{\bw})^{\dag}R_3^{\HH}\be_k\right)+2{\rm Re}(\iota_k)+| q_k|^2
\end{align*}
giving $2{\rm Re}(\iota_k)=2{\rm Re}\left(\bc(\bw)^{\HH}B_w (A_{\bw}-d_2(\bw)B_{\bw})^{\dag}R_3^{\HH}\be_k\right)-| q_k|^2.$
Now, plug it and \eqref{eq:hess2} into \eqref{eq:hess1} to have 
\begin{align} \nonumber
  [\nabla^2 d_2(\bw)]_{jk} 
=&\frac{\partial [\nabla  d_2(\bw)]_j}{\partial w_k}=\be_j^{\T}\nabla^2 d_2(\bw)\be_k\\\nonumber
 \label{eq:hess3}
=&-2\be_j^{\T}{\rm Re}\left(\left( {R}_1^{\HH} E_1 -d_2(\bw) {R}_2^{\HH}E_2\right) (A_{\bw}-d_2(\bw)B_{\bw})^{\dag}R_3^{\HH}\be_k\right)\\\nonumber
&+2\be_j^{\T}{\rm Re}\left(\left( {R}_1^{\HH} E_1 -d_2(\bw) {R}_2^{\HH}E_2\right)  \bc(\bw)\iota_k\right)- [\nabla  d_2(\bw)]_k| q_j|^2\\\nonumber
=&-2\be_j^{\T}{\rm Re}\left(\left( {R}_1^{\HH} E_1 -d_2(\bw) {R}_2^{\HH}E_2\right) (A_{\bw}-d_2(\bw)B_{\bw})^{\dag}R_3^{\HH}\be_k\right)\\\nonumber
&+2[\nabla d_2(\bw)]_j {\rm Re}\left( \iota_k\right)- [\nabla  d_2(\bw)]_k| q_j|^2\\\nonumber
=&-2\be_j^{\T}{\rm Re}\left(\left( {R}_1^{\HH} E_1 -d_2(\bw) {R}_2^{\HH}E_2\right) (A_{\bw}-d_2(\bw)B_{\bw})^{\dag}R_3^{\HH}\be_k\right)\\\nonumber
&+2[\nabla d_2(\bw)]_j  {\rm Re}\left(\bc(\bw)^{\HH}B_{\bw} (A_{\bw}-d_2(\bw)B_{\bw})^{\dag}R_3\be_k\right)\\\nonumber
&-[\nabla d_2(\bw)]_j | q_k|^2- [\nabla  d_2(\bw)]_k| q_j|^2\\\nonumber
=&-2\be_j^{\T}{\rm Re}\left(\left({R}_1^{\HH} E_1 -d_2(\bw) {R}_2^{\HH}E_2 -\nabla d_2(\bw)\bc(\bw)^{\HH}B_{\bw}\right) (A_{\bw}-d_2(\bw)B_{\bw})^{\dag}R_3^{\HH}\be_k\right)\\\nonumber
&-[\nabla d_2(\bw)]_j | q_k|^2- [\nabla  d_2(\bw)]_k| q_j|^2\\\nonumber
=&-\be_j^{\T}\left({\rm Re}\left(2R_3 (A_{\bw}-d_2(\bw)B_{\bw})^{\dag}R_3^{\HH}\right)-\nabla d_2(\bw)(|\bq|^2)^{\T}-|\bq|^2 (\nabla d_2(\bw))^{\T}\right)\be_k.
\end{align}
{With this formulation for the $(j,k)$th entry of $\nabla^2 d_2(\bw)$,  \eqref{eq:Hess} follows.} $\hfill\square$

\vskip 2mm

\noindent {\it Proof of Corollary \ref{cor:Hess2}}. 
The new formulation of $\nabla^2 d_2(\bw)$ can be obtained from the proof for Proposition \ref{prop:grad} (ii).
Using the QR factorizations 
$W^{\frac12}\Phi=Q_qR_q$ and $W^{\frac12}\Psi=Q_pR_p$,  we first have 
$$
A_{\bw}-d_2(\bw)B_{\bw}=R_{pq}^{\HH}D R_{pq},~~R_{pq}^{-\HH}R_3^{\HH}=\what R_3^{\HH},
$$
where $R_{pq}:=\diag(R_p,R_q)$.
The key observation is to note that the relation \eqref{eq:partialcwk} leads to 
\begin{align*}
D \left(R_{pq}\frac{\partial \bc(\bw)}{\partial w_k}\right)=-R_{pq}^{-\HH}R_3^{\HH}\be_k=-\what R_3^{\HH} \be_k.
\end{align*}
Similarly, from  $0=\bc(\bw)^{\HH}R_3^{\HH}\be_k=(R_{pq} \bc(\bw))^{\HH}R_{pq}^{-\HH}R_3^{\HH}\be_k=(R_{pq} \bc(\bw))^{\HH}\what R_3^{\HH}\be_k$ and $0=R^{-\HH}_{pq}(A_{\bw}-d_2(\bw)B_{\bw})\bc(\bw)= D(R_{pq}\bc(\bw))$; hence $R_{pq}\frac{\partial \bc(\bw)}{\partial w_k}$ can be expressed as
\begin{equation}\nonumber
R_{pq}\frac{\partial \bc(\bw)}{\partial w_k}=-D^{\dag}\what R_3^{\HH}\be_k+\iota_k R_{pq}\bc(\bw),~~{\rm for~some ~}\iota_k\in \bbC,
\end{equation}
or equivalently, 
$
 \frac{\partial \bc(\bw)}{\partial w_k}=-R_{pq}^{-1}D^{\dag}\what R_3^{\HH}\be_k+\iota_k  \bc(\bw),~~{\rm for~some ~}\iota_k\in \bbC.
$
The rest  follows similarly {to}  the proof for Proposition \ref{prop:grad} (ii) and we omit the details. $\hfill\square$

{
\section*{Acknowledgements}
The authors would like to thank the anonymous referees for their careful reading, useful comments and suggestions to improve the presentation of the paper.}

\def\noopsort#1{}\def\l{\char32l}\def\v#1{{\accent20 #1}}
  \let\^^_=\v\def\hbk{hardback}\def\pbk{paperback}
\providecommand{\href}[2]{#2}
\providecommand{\arxiv}[1]{\href{http://arxiv.org/abs/#1}{arXiv:#1}}
\providecommand{\url}[1]{\texttt{#1}}
\providecommand{\urlprefix}{URL }


\begin{thebibliography}{10}
\small
\bibitem{akhi:1990}
\newblock N.~I. Akhiezer,
\newblock \emph{Elements of the Theory of Elliptic Functions},
\newblock Transl. Math. Monogr. 79, American Mathematical Society, Providence,
  RI, 1990.

\bibitem{bama:1970}
\newblock I.~Barrodale and J.~Mason,
\newblock Two simple algorithms for discrete rational approximation,
\newblock \emph{Math. Comp.}, \textbf{24} (1970), 877--891.

\bibitem{bapr:1972}
\newblock I.~Barrodale, M.~J.~D. Powell and F.~D.~K. Roberts,
\newblock The differential correction algorithm for rational $\ell_\infty$
  approximation,
\newblock \emph{SIAM J. Numer. Anal.}, \textbf{9} (1972), 493--504.

\bibitem{begu:2015}
\newblock M.~Berljafa and S.~G\"uttel,
\newblock Generalized rational {K}rylov decompositions with an application to
  rational approximation,
\newblock \emph{SIAM J. Matrix Anal. Appl.}, \textbf{36} (2015), 894--916,
\newblock \urlprefix\url{https: //doi.org/10.1137/140998081}.

\bibitem{begu:2017}
\newblock M.~Berljafa and S.~G\"{u}ttel,
\newblock The {RKFIT} algorithm for nonlinear rational approximation,
\newblock \emph{SIAM J. Sci. Comput.}, \textbf{39} (2017), A2049--A2071,
\newblock \urlprefix\url{https://doi.org/10.1137/15M1025426}.

\bibitem{boyd:2004}
\newblock S.~Boyd and L.~Vandenberghe,
\newblock \emph{Convex Optimization},
\newblock Cambridge University Press, 2004.

\bibitem{brnt:2021}
\newblock P.~D. Brubeck, Y.~Nakatsukasa and L.~N. Trefethen,
\newblock Vandermonde with {A}rnoldi,
\newblock \emph{SIAM Rev.}, \textbf{63} (2021), 405--415.

\bibitem{chlo:1963}
\newblock E.~W. Cheney and H.~L. Loeb,
\newblock Two new algorithms for rational approximation,
\newblock \emph{Numer. Math.}, \textbf{3} (1961), 72--75.

\bibitem{clin:1972}
\newblock A.~K. Cline,
\newblock Rate of convergence of {L}awson's algorithm,
\newblock \emph{Math. Comp.}, \textbf{26} (1972), 167--176.

\bibitem{coop:2007}
\newblock P.~Cooper,
\newblock \emph{Rational approximation of discrete data with asymptotic
  behaviour},
\newblock PhD thesis, University of Huddersfield, 2007.

\bibitem{decz:1974}
\newblock A.~Deczky,
\newblock {Equiripple and minimax (Chebyshev) approximations for recursive
  digital filters},
\newblock \emph{IEEE Trans. Acoust., Speech, Signal Processing}, \textbf{22}
  (1974), 98--111.

\bibitem{dges:1999}
\newblock J.~W. Demmel, M.~Gu, S.~Eisenstat, I.~Slapni\v{c}ar, K.~Veseli{\'c}
  and Z.~Drma\v{c},
\newblock Computing the singular value decomposition with high relative
  accuracy,
\newblock \emph{Linear Algebra Appl.}, \textbf{299} (1999), 21--80.

\bibitem{demm:1992}
\newblock J.~W. Demmel and K.~Veseli{\'c},
\newblock {Jacobi's} method is more accurate than {QR},
\newblock \emph{SIAM J. Matrix Anal. Appl.}, \textbf{13} (1992), 1204--1245.

\bibitem{drht:2014}
\newblock T.~A. Driscoll, N.~Hale and L.~N. Trefethen,
\newblock \emph{{Chebfun User's Guide}},
\newblock Pafnuty Publications, Oxford, 2014,
\newblock See also www.chebfun.org.

\bibitem{elwi:1976}
\newblock S.~Ellacott and J.~Williams,
\newblock Linear {C}hebyshev approximation in the complex plane using
  {L}awson's algorithm,
\newblock \emph{Math. Comp.}, \textbf{30} (1976), 35--44.

\bibitem{elli:1978}
\newblock G.~H. Elliott,
\newblock \emph{The construction of {Chebyshev} approximations in the complex
  plane},
\newblock PhD thesis, Facultyof Science (Mathematics), University of London,
  1978.

\bibitem{fint:2018}
\newblock S.-I. Filip, Y.~Nakatsukasa, L.~N. Trefethen and B.~Beckermann,
\newblock Rational minimax approximation via adaptive barycentric
  representations,
\newblock \emph{SIAM J. Sci. Comput.}, \textbf{40} (2018), A2427--A2455.

\bibitem{govl:2013}
\newblock G.~H. Golub and C.~F. {Van Loan},
\newblock \emph{Matrix Computations},
\newblock 4th edition,
\newblock Johns Hopkins University Press, Baltimore, Maryland, 2013.

\bibitem{gogu:2021}
\newblock I.~V. Gosea and S.~G\"{u}ttel,
\newblock Algorithms for the rational approximation of matrix-valued functions,
\newblock \emph{SIAM J. Sci. Comput.}, \textbf{43} (2021), A3033--A3054,
\newblock \urlprefix\url{https://doi.org/10.1137/20M1324727}.

\bibitem{guse:1999}
\newblock B.~Gustavsen and A.~Semlyen,
\newblock Rational approximation of frequency domain responses by vector
  fitting,
\newblock \emph{IEEE Trans. Power Deliv.}, \textbf{14} (1999), 1052--1061.

\bibitem{gutk:1983}
\newblock M.~H. Gutknecht,
\newblock On complex rational approximation. {Part I}: The characterization
  problem,
\newblock in \emph{{Computational Aspects of Complex Analysis (H. Werneret
  at.,eds.). Dordrecht : Reidel}}, 1983,
\newblock 79--101.

\bibitem{gust:1983}
\newblock M.~H. Gutknecht, J.~O. Smith and L.~N. Trefethen,
\newblock The {C}aratheodory-{F}\'ejer method for recursive, digital filter
  design,
\newblock \emph{IEEE Trans. Acoust., Speech, Signal Processing},
  \textbf{ASSP-31} (1983), 1417--1426.

\bibitem{hoka:2020}
\newblock J.~M. Hokanson,
\newblock Multivariate rational approximation using a stabilized
  {S}anathanan-{K}oerner iteration, 2020,
\newblock \urlprefix\url{arXiv:2009.10803v1}.

\bibitem{isth:1993}
\newblock M.-P. Istace and J.-P. Thiran,
\newblock On computing best {C}hebyshev complex rational approximants,
\newblock \emph{Numer. Algorithms}, \textbf{5} (1993), 299--308.

\bibitem{itna:2018}
\newblock S.~Ito and Y.~Nakatsukasa,
\newblock Stable polefinding and rational least-squares fitting via
  eigenvalues,
\newblock \emph{Numer. Math.}, \textbf{139} (2018), 633--682.

\bibitem{laws:1961}
\newblock C.~L. Lawson,
\newblock \emph{Contributions to the Theory of Linear Least Maximum
  Approximations},
\newblock PhD thesis, UCLA, USA, 1961.

\bibitem{lilb:2013}
\newblock X.~Liang, R.-C. Li and Z.~Bai,
\newblock Trace minimization principles for positive semi-definite pencils,
\newblock \emph{Linear Algebra Appl.}, \textbf{438} (2013), 3085--3106.

\bibitem{loeb:1957}
\newblock H.~L. Loeb,
\newblock \emph{On rational fraction approximations at discrete points},
\newblock Technical report, Convair Astronautics, 1957,
\newblock Math. Preprint \#9.

\bibitem{dips:2022b}
\newblock R.~D. Mill\'an, V.~Peiris, N.~Sukhorukova and J.~Ugon,
\newblock Application and issues in abstract convexity, 2022,
\newblock \urlprefix\url{arXiv:2202.09959}.

\bibitem{dips:2022a}
\newblock R.~D. Mill\'an, V.~Peiris, N.~Sukhorukova and J.~Ugon,
\newblock Multivariate approximation by polynomial and generalised rational
  functions,
\newblock \emph{Optimization}, \textbf{71} (2022), 1171--1187.

\bibitem{nase:2018}
\newblock Y.~Nakatsukasa, O.~S\`ete and L.~N. Trefethen,
\newblock The {AAA} algorithm for rational approximation,
\newblock \emph{SIAM J. Sci. Comput.}, \textbf{40} (2018), A1494--A1522.

\bibitem{natr:2020}
\newblock Y.~Nakatsukasa and L.~N. Trefethen,
\newblock An algorithm for real and complex rational minimax approximation,
\newblock \emph{SIAM J. Sci. Comput.}, \textbf{42} (2020), A3157--A3179.

\bibitem{nest:2018}
\newblock Y.~Nesterov,
\newblock \emph{Lectures on Convex Optimization},
\newblock 2nd edition,
\newblock Springer, 2018.

\bibitem{nowr:2006}
\newblock J.~Nocedal and S.~Wright,
\newblock \emph{Numerical Optimization},
\newblock 2nd edition,
\newblock Springer, New York, 2006.

\bibitem{rice:1969}
\newblock J.~R. Rice,
\newblock \emph{The approximation of functions},
\newblock Addison-Wesley publishing Company, 1969.

\bibitem{rutt:1985}
\newblock A.~Ruttan,
\newblock A characterization of best complex rational approximants in a
  fundamental case,
\newblock \emph{Constr. Approx.}, \textbf{1} (1985), 287--296.

\bibitem{sava:1977}
\newblock E.~B. Saff and R.~S. Varga,
\newblock Nonuniqueness of best approximating complex rational functions,
\newblock \emph{Bull. Amer. Math. Soc.}, \textbf{83} (1977), 375--377.

\bibitem{sako:1963}
\newblock C.~K. Sanathanan and J.~Koerner,
\newblock Transfer function synthesis as a ratio of two complex polynomials,
\newblock \emph{IEEE T. Automat. Contr.}, \textbf{8} (1963), 56--58,
\newblock \urlprefix\url{https://doi.org/10.1109/TAC.1963.1105517}.

\bibitem{stah:1993}
\newblock G.~Stahl,
\newblock Best uniform approximation of |x| on [-1,1],
\newblock \emph{Sb. Math.}, \textbf{76} (1993), 461--487.

\bibitem{this:1993}
\newblock J.-P. Thiran and M.-P. Istace,
\newblock Optimality and uniqueness conditions in complex rational {C}hebyshev
  approximation with examples,
\newblock \emph{Constr. Approx.}, \textbf{9} (1993), 83--103.

\bibitem{tref:2019a}
\newblock L.~N. Trefethen,
\newblock \emph{{Approximation Theory and Approximation Practice, {E}xtended
  Edition}},
\newblock SIAM, 2019.

\bibitem{will:1972}
\newblock J.~Williams,
\newblock Numerical {C}hebyshev approximation in the complex plane,
\newblock \emph{SIAM J. Numer. Anal.}, \textbf{9} (1972), 638--649.

\bibitem{will:1979}
\newblock J.~Williams,
\newblock Characterization and computation of rational {C}hebyshev
  approximations in the complex plane,
\newblock \emph{SIAM J. Numer. Anal.}, \textbf{16} (1979), 819--827.

\bibitem{wulb:1980}
\newblock D.~E. Wulbert,
\newblock On the characterization of complex rational approximations,
\newblock \emph{Illinois J. Math.}, \textbf{24} (1980), 140--155.

\bibitem{yazz:2023}
\newblock L.~Yang, L.-H. Zhang and Y.~Zhang,
\newblock The ${L}q$-weighted dual programming of the linear {C}hebyshev
  approximation and an interior-point method,
\newblock  \emph{Adv. in Comput. Math.}, \textbf{50:80} (2024).

\bibitem{zhsl:2023}
\newblock L.-H. Zhang, Y.~Su and R.-C. Li,
\newblock Accurate polynomial fitting and evaluation via {A}rnoldi,
\newblock \emph{Numerical Algebra, Control and Optimization}, \textbf{to
  appear},
\newblock Doi:10.3934/naco.2023002.

\end{thebibliography}

%
\end{document}